\documentclass[aps,prl,reprint,superscriptaddress,nofootinbib]{revtex4-2}

\usepackage{amsmath, amssymb,amsthm,mathtools}
\usepackage{booktabs}
\usepackage{microtype}
\usepackage[hidelinks]{hyperref}
\usepackage{xcolor}
\usepackage{color}

\usepackage{tikz}
\usetikzlibrary{arrows.meta}

\allowdisplaybreaks

\newcommand{\E}{\mathbb{E}}
\newcommand{\Var}{\operatorname{Var}}
\newcommand{\Cov}{\operatorname{Cov}}
\newcommand{\KL}{D_{\mathrm{KL}}}
\newcommand{\Z}{\mathbb{Z}}

\theoremstyle{plain}
\newtheorem{theorem}{Theorem}
\newtheorem{lemma}{Lemma}
\newtheorem{proposition}{Proposition}

\theoremstyle{definition}
\newtheorem{definition}{Definition}
\newtheorem{remark}{Remark}

\newtheorem{assumption}{Assumption}

\begin{document}

\title{Fluctuation impossibility results for stochastic burst networks}

\author{David F. Anderson}
\affiliation{Department of Mathematics, University of Wisconsin--Madison, Madison, Wisconsin 53706, USA}

\date{\today}

\begin{abstract}
Stochastic reaction networks, arising in biochemical systems and other
nonequilibrium settings, often involve components at low copy number, where
individual production and degradation events generate substantial
fluctuations.  In a 2019 \textit{Physical Review Letters} article, Yan,
Hilfinger, Vinnicombe, and Paulsson conjectured that, for networks with linear
degradation and arbitrary cross-regulatory production rates, feedback cannot
suppress the stationary fluctuations for each component below the fluctuations of its constant-rate
counterpart~\cite{yan2019kinetic}.  With components indexed by $i$, their formulation allows production events
to create a random number $K_i$ of molecules: the unit-birth case has
$K_i\equiv1$, the biologically important burst model takes $K_i$ to be
geometrically distributed, but more general positive integer-valued burst laws
are also permitted.

The conjecture was recently proved in the case of unit births, i.e., $K_i \equiv 1$ \cite{RKH2025,anderson2026noise}.
We show here that the conjecture is \textit{false in general}  by
constructing a two-component network with bounded production rates and burst
sizes in $\{1,2\}$ for which both stationary Fano factors lie below their common
constant-rate baseline.   However, we then prove the conjecture for positive geometric bursts, the canonical burst model in stochastic gene expression.  Specifically, for arbitrary regulatory architecture and arbitrary nonlinear cross-regulatory production rates, we prove an exact weighted tradeoff that rules out simultaneous suppression of every component below its geometric-burst
baseline.  The geometric theorem includes the unit-birth model as a degenerate
case and therefore subsumes the previously proved unit-birth result.
 We also prove a complementary structural impossibility result for  arbitrary positive integer-valued burst laws with finite second moments: if the activating and inhibiting interactions have a
globally consistent sign structure, in the sense that every cycle of the regulatory interaction graph contains an
even number of negative interactions, then every component individually
satisfies $F_{X_i}\ge B_i$, where $F_{X_i}$ is its stationary Fano factor and
$B_i$ is its constant-rate burst baseline.
\end{abstract}

\maketitle

Reaction and interaction networks are used throughout the sciences to model
systems ranging from biochemical regulation and chemical kinetics to
population dynamics, ecological communities, and epidemics \cite{anderson2015stochastic, allen2010introduction}.  When the
relevant copy numbers are low, discrete stochastic models are often the
natural mathematical description.  In this regime, individual production and
degradation events can generate substantial fluctuations, which may propagate
through the network and affect switching, timing, information transmission,
and the reliability of downstream responses
\cite{gillespie1977exact,mcadams1997stochastic,thattai2001intrinsic,
elowitz2002stochastic,paulsson2004summing,raj2008nature,
purvis2013encoding}.  This raises a fundamental question: \textit{even allowing an
arbitrarily complicated interaction network and arbitrary nonlinear
regulatory mechanisms, is it possible to simultaneously suppress the intrinsic fluctuations
of every component?}    Many systems of interest operate far from equilibrium, so any universal
limitation of this kind would be especially striking.

Suppressing fluctuations in a single selected component via regulation is not, by itself,
difficult.  For example, an upward fluctuation in a component $X_1$ may
increase the production of another component $X_2$, which in turn represses
the production of $X_1$ and pushes it back toward its mean.  More generally, a
network may use several comparatively noisy components to stabilize one
component, or even a chosen collection of components.  The substantially more
difficult question is whether a sufficiently sophisticated network can
suppress the fluctuations of \emph{every} component simultaneously.  Suppose
one allows an arbitrary number of components, an arbitrary interaction
topology, and arbitrary nonlinear regulatory functions.  Can such an unrestricted design reduce the fluctuations of every component
below its natural baseline, or must reducing fluctuations in some components
necessarily increase them in others? 

To formulate this question precisely, one first needs a baseline mathematical
model.  For concreteness, we use the language of biochemical reaction networks
and refer to the components as molecular species, although the results apply
more generally to stochastic population models.  Index the species by $i$, and
let $X_i$ denote the count of species $i$.  Each molecule degrades independently
at rate $1/\tau_i$, for some $\tau_i>0$, so degradation events occur at total
rate $X_i/\tau_i$ and decrease the count by one.
Production, in contrast, occurs in bursts.  Let $K_i$ be a positive
integer-valued random variable,
\[
    K_i\in\{1,2,\ldots\},
    \qquad
    p_{i,k}:=\mathbb P(K_i=k).
\]
Production events occur at constant rate $\lambda_i>0$, and each such event
adds an independent copy of $K_i$ to the count.  Thus the model is
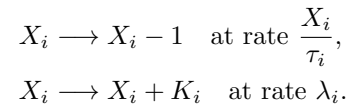
\begin{align*}
        &X_i\longrightarrow X_i-1
        \quad\text{at rate }\frac{X_i}{\tau_i},\\
        &X_i\longrightarrow X_i+K_i
        \quad\text{at rate }\lambda_i.
\end{align*}

The choice $K_i\equiv1$ gives the classical unit-birth--death process, whose
stationary distribution is Poisson with mean $\lambda_i\tau_i$.  Standard
two-stage models of gene expression describe transcription and translation
explicitly \cite{thattai2001intrinsic,paulsson2004summing}.  When mRNA is
short-lived relative to protein, eliminating the mRNA yields effectively
instantaneous protein bursts with a geometric size distribution
\cite{shahrezaei2008analytical}.  Bursty production has also been observed and
characterized across bacterial and eukaryotic systems
\cite{golding2005real,sanchez2013genetic,chong2014mechanism,
fukaya2016enhancer}.  More generally, the same mathematical structure appears in queueing theory
as an infinite-server queue with batch arrivals, in which customers arrive in
random batches and depart after independent exponential service times
\cite{shanbhag1966infinite,chaudhry1983first}.  We therefore allow any positive
integer-valued burst law with finite second moment.

Variance is the natural measure of the size of stationary fluctuations, but
variance alone does not provide a fair comparison across systems with
different mean abundances.  Following Ref.~\cite{yan2019kinetic}, we compare
systems at fixed stationary mean; for fixed burst law and degradation
timescale, this is equivalent to fixing the mean production rate.  The corresponding normalized measure is the Fano factor
\[
        F_{X_i}:=
        \frac{\Var_\pi(X_i)}{\E_\pi[X_i]},
\]
where $\pi$ denotes the stationary distribution of the process under
consideration.
For the constant-rate production process above, the stationary moment
identities proved in Lemma~\ref{lem:S-moment-identities} of the Supplemental
Material give
\begin{equation}
        F_{X_i}
        =
        B_i,
        \qquad
        B_i
        :=
        \frac{\E[K_i^2]+\E[K_i]}{2\E[K_i]}.
        \label{eq:baseline-main}
\end{equation}
Thus $B_i$ is the natural fluctuation baseline determined by the burst law.  For unit births, $B_i=1$, since the stationary distribution is Poisson.  
 If
$K_i$ is geometric on the positive integers with mean $\kappa_i$, then
$B_i=\kappa_i$ (see \eqref{eq:positive-geometric-main} and \eqref{eq:geometric_stuff}, below).

Now consider a network
\[
        X=(X_1,\dots,X_N)
\]
of such components.  Biological systems are rarely collections of independent
production--degradation processes.  One molecular species may regulate another
through transcriptional regulation, signaling pathways, sequestration,
enzymatic activity, or other interaction mechanisms
\cite{shen2002network,tyson2003sniffers,purvis2013encoding,
buchler2009protein}.  We therefore allow component $i$ to be produced at a
state-dependent rate $f_i(X_{-i})$, where $X_{-i}$ is the vector obtained from
$X$ by removing its $i$th component.  The model then becomes
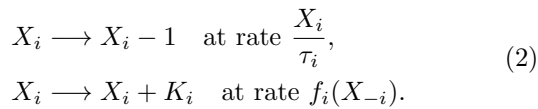
\begin{align}
\begin{split}
         X_i & \longrightarrow X_i-1
        \quad\text{at rate }\frac{X_i}{\tau_i},\\
        X_i & \longrightarrow X_i+K_i
        \quad\text{at rate }f_i(X_{-i}).
        \end{split}
        \label{eq:model-reactions-main}
\end{align}
The associated interaction graph has a directed edge $j\to i$ whenever
$f_i$ depends nontrivially on $x_j$.  When that dependence is monotone, we
label the edge positive or negative according to whether $f_i$ is
increasing or decreasing in $x_j$.
The restriction that $f_i$ does not depend directly on $X_i$ means that
feedback to component $i$ must pass through at least one other component.
Apart from this restriction, the interaction graph and the rate functions are
arbitrary: dependencies may be nonlinear and nonmonotone.  A representative
signed interaction graph is shown in Fig.~\ref{fig:interaction-network}.

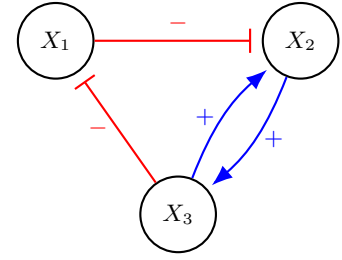
\begin{figure}[t]
\centering
\begin{tikzpicture}[
    scale=1.35,
    species/.style={
        draw,
        circle,
        thick,
        minimum size=10mm,
        inner sep=1pt
    },
    activation/.style={
        -{Latex[length=2.8mm]},
        thick,
        shorten >=2pt
    },
    inhibition/.style={
        -{Bar[width=3.2mm]},
        thick,
        shorten >=4pt
    }
]

\node[species] (X1) at (0,1.7) {$X_1$};
\node[species] (X2) at (2.4,1.7) {$X_2$};
\node[species] (X3) at (1.2,0) {$X_3$};

\draw[inhibition,red]
    (X1) -- node[midway,above,font=\small] {$-$} (X2);

\draw[activation,blue,bend left=14]
    (X2) to node[midway,right,font=\small] {$+$} (X3);

\draw[activation,blue,bend left=14]
    (X3) to node[midway,left,font=\small] {$+$} (X2);

\draw[inhibition,red]
    (X3) -- node[midway,left,font=\small] {$-$} (X1);

\end{tikzpicture}

\caption{A representative three-component regulatory network.  Arrowheads
denote activating interactions and T-bars denote inhibiting interactions.}
\label{fig:interaction-network}
\end{figure}

The fundamental design question can now be stated precisely: is it possible to
choose the interaction network and the functions $f_i$ so that
\begin{equation}
        F_{X_i}<B_i
        \qquad
        \text{for every }i\in\{1,\dots,N\}?
        \label{eq:simultaneous-suppression-main}
\end{equation}
The point is not whether feedback can suppress one component, or several
chosen components.  Rather, the question is whether an arbitrarily complicated
network can suppress \emph{all} components simultaneously.

In fact, one can ask the stronger impossibility question of whether it is
possible merely to keep every component at or below its baseline while
strictly suppressing at least one:
\begin{align}
\begin{split}
        F_{X_i} &\le B_i  \qquad\text{for every }i,
        \qquad\text{with}\\
        F_{X_j} &<B_j
        \qquad\text{for at least one }j.
        \label{eq:weak-simultaneous-suppression-main}
        \end{split}
\end{align}

Yan, Hilfinger, Vinnicombe, and Paulsson posed this question in a 2019
\textit{Physical Review Letters} article and conjectured that the answer is no
\cite{yan2019kinetic}.  Their conjecture permits arbitrary cross-regulatory
topology, arbitrary component lifetimes, arbitrary nonlinear production
functions, and general burst-size distributions.  In the notation above, it
asserts that every stationary network in this class must contain at least one
component for which 
\[
        F_{X_i}> B_i,
\]
if another satisfies $F_{X_j} < B_j$.
The strength of the conjecture is precisely its independence from the details
of the regulatory architecture.  It proposes an impossibility result: no matter
how many components are used, how they are connected, or how the nonlinear
production rates are chosen, the fluctuations of every component cannot be
simultaneously reduced below their constant-rate levels. 

Yan et al.\ proved the corresponding statement in the high-copy-number
linear-noise approximation and supported the exact finite-copy conjecture with
extensive numerical experiments
\cite{yan2019kinetic}.  More recently, Ripsman, Kell, and Hilfinger developed
an exact stationary information-flow argument for the unit-birth model
\cite{RKH2025}, and the required infinite-state-space justification was
provided in Ref.~\cite{anderson2026noise}.  The conjecture is therefore known
to hold when $K_i\equiv1$ for every component.

The present work asks whether the same impossibility principle survives when
production is genuinely bursty.  We show that the answer depends on the burst
distribution.  First, the unrestricted conjecture is false: we construct an analytic family of two-component networks with bounded
production rates and burst sizes in $\{1,2\}$
 for which
\[
        F_{X_1}<B_1,
        \qquad
        F_{X_2}<B_2.
\]
Thus even a simple nongeometric burst law can permit simultaneous suppression.

However, our main positive result,
Theorem~\ref{thm:geometric-main}, proves the conjecture for positive geometric
bursts, the central burst model in stochastic gene expression.  If component
$i$ has geometric burst size with mean $\kappa_i$ and lifetime $\tau_i$, then
$B_i=\kappa_i$ and
\begin{equation}
        \sum_{i=1}^N
        \frac{F_{X_i}-\kappa_i}{\tau_i\kappa_i}
        \ge0.
        \label{eq:geometric-tradeoff-preview}
\end{equation}
Consequently, if any component lies below its geometric-burst baseline, then
at least one other component must lie above its baseline.
Uniform suppression below the geometric-burst baselines is therefore
impossible.  Since $K_i\equiv1$ is the degenerate positive-geometric case
$\kappa_i=1$, the theorem also subsumes the previously proved unit-birth
result.

A complementary impossibility mechanism comes from the interaction structure.
Theorem~\ref{thm:signed-main} shows that, for arbitrary positive burst laws,
if every regulatory dependence has a fixed sign and every cycle contains an
even number of negative interactions, then
\[
        F_{X_i}\ge B_i
        \qquad\text{for every }i.
\]
Thus this regime is as far from simultaneous suppression as possible: no
component can lie below its own baseline.  The network in
Figure~\ref{fig:interaction-network} illustrates this condition: arrowheads and
T-bars represent positive and negative interactions, respectively, and every
cycle contains an even number of negative interactions, making the signs
globally consistent.

We now formulate the model and the results precisely.

\emph{Model and assumptions.---}
 Let
$X=(X_1,\ldots,X_N)$ be a continuous-time Markov chain on
$\mathbb Z_{\geq 0}^N$.  Its generator, acting on bounded test functions
$\varphi$, is
\begin{align}
(L\varphi)(x)
&=
\sum_{i=1}^N f_i(x_{-i})
\sum_{k\geq1}p_{i,k}
\bigl[\varphi(x+ke_i)-\varphi(x)\bigr]
\notag\\
&\quad+
\sum_{i=1}^N\frac{x_i}{\tau_i}
\bigl[\varphi(x-e_i)-\varphi(x)\bigr],
\label{eq:generator-main}
\end{align}
where the death term is omitted when $x_i=0$.

Equivalently, writing
$P_t(x)=\mathbb P(X(t)=x)$, the law of the process satisfies the chemical
master equation
\begin{align}
\frac{d}{dt}P_t(x)
&=
\sum_{i=1}^N f_i(x_{-i})
\left[
\sum_{k=1}^{x_i}p_{i,k}P_t(x-ke_i)-P_t(x)
\right]
\notag\\
&\quad+
\sum_{i=1}^N
\left[
\frac{x_i+1}{\tau_i}P_t(x+e_i)
-\frac{x_i}{\tau_i}P_t(x)
\right].
\label{eq:master-equation-main}
\end{align}
We use the generator formulation throughout the proofs.  Sample paths may be
generated using Gillespie's stochastic simulation algorithm, with an
independent copy of $K_i$ sampled at each production event
\cite{gillespie1977exact,anderson2015stochastic}.

We work under the standing assumptions stated precisely in the Supplemental
Material (see Assumption~\ref{ass:S-standing}). In particular, the chain is irreducible and admits a stationary
distribution $\pi$, the relevant second moments are finite, the stationary
mean production rates
\[
        \mu_i:=\E_\pi[f_i(X_{-i})]
\]
are finite and positive, and the total jump rate grows at most linearly.  The
Supplemental Material also gives checkable Foster--Lyapunov conditions under
which these assumptions hold (see Theorem~\ref{thm:S-sufficient-conditions}). 

Under these assumptions, Lemma~\ref{lem:S-moment-identities} of the
Supplemental Material gives
\begin{equation}
        \frac{\E_\pi[X_i]}{\tau_i}
        =
        \kappa_i\mu_i,
        \qquad
        F_{X_i}-B_i
        =
        \frac{\Cov_\pi(f_i(X_{-i}),X_i)}{\mu_i},
\label{eq:fano-main}
\end{equation}
where $\kappa_i=\E[K_i]$ and $B_i$ is the constant-rate baseline defined in
\eqref{eq:baseline-main}.  Thus component $i$ lies below its constant-rate
baseline precisely when its production rate is negatively correlated with its
copy number.

\emph{A  two-component counterexample.---}
Consider $N=2$, with $\tau_1=\tau_2=1$, and, for each of $i \in \{1,2\},$ let the two components have
the common burst distribution
\begin{equation}
        \mathbb P(K_i=1)=1-r,
        \qquad
        \mathbb P(K_i=2)=r,
\label{eq:counter-burst-main}
\end{equation}
where $r\in(1/2,1]$.  Then for each $i$,
\[
        \kappa_i= 1+r, \quad \E[K_i^2] = 1 + 3r, \quad B_i = B(r) := \frac{1+2r}{1+r}.
\]

For $\varepsilon>0$, define the bounded production rates
\begin{equation}
f_1^{(\varepsilon)}(n)
=
\begin{cases}
\varepsilon, & n=0,\\
\frac65 \varepsilon, & n=1,\\
\frac95 \varepsilon, & n\ge2,
\end{cases}
\qquad
f_2^{(\varepsilon)}(n)
=
\begin{cases}
\varepsilon, & n=0,\\
\frac45 \varepsilon, & n=1,\\
\frac15 \varepsilon, & n\ge2.
\end{cases}
\label{eq:analytic-counter-rates-main}
\end{equation}
Thus $f_1^{(\varepsilon)}$ is increasing, while
$f_2^{(\varepsilon)}$ is decreasing.

This model is fully analyzed in 
Section~\ref{sec:S-counterexamples} of the Supplemental Material.  In particular, Proposition~\ref{prop:S-analytic-counterexample} 
gives, for each $i \in\{1,2\}$,
\begin{equation}
        F_{X_i} =  B(r)  - \frac{2r(2r-1)}{75}\varepsilon^2 +  O(\varepsilon^3), \quad \text{as} \quad  \varepsilon\downarrow0.
\label{eq:analytic-counter-fano-main}
\end{equation}
Since $r>1/2$, the coefficient of $\varepsilon^2$ is strictly negative.
Consequently, for every fixed $r\in(1/2,1]$, there exists
$\varepsilon_0(r)>0$ such that
\[
        F_{X_1}<B(r),
        \qquad
        F_{X_2}<B(r),
\]
when $0<\varepsilon<\varepsilon_0(r)$.
Hence, the unrestricted conjecture of \cite{yan2019kinetic} is false.

\emph{Geometric bursts.---}
We now specialize to positive geometric bursts, for $k \ge 1,$
\begin{equation}
        \mathbb P(K_i=k)
        =
        p_i(1-p_i)^{k-1},
        \qquad
        p_i\in(0,1].
\label{eq:positive-geometric-main}
\end{equation}
Their means and second moments satisfy
\begin{align}
\label{eq:geometric_stuff}
        \kappa_i=\frac1{p_i},
        \qquad
        \E[K_i^2]=\frac{2-p_i}{p_i^2},
\end{align}
so the general baseline in \eqref{eq:baseline-main} and  \eqref{eq:fano-main} reduces to
\[
        B_i=\kappa_i.
\]
Our main positive result is the following exact impossibility theorem.  

\begin{theorem}[Geometric-burst tradeoff]
\label{thm:geometric-main}
Suppose Assumption~\ref{ass:S-standing} holds and every $K_i$ has the positive
geometric law \eqref{eq:positive-geometric-main}.  Then
\begin{equation}
        \sum_{i=1}^N
        \frac{F_{X_i}-\kappa_i}{\tau_i\kappa_i}
        \ge0.
\label{eq:geom-tradeoff-main}
\end{equation}
Consequently, if one component lies below its geometric-burst baseline, then
at least one other component must lie strictly above its own.
\end{theorem}
The two-component implication is illustrated in
Fig.~\ref{fig:geometric-impossibility}.

The endpoint $p_i=1$ gives $K_i\equiv1$, so
Theorem~\ref{thm:geometric-main} contains the unit-birth theorem of
Refs.~\cite{RKH2025,anderson2026noise}.  Many gene-expression models instead
use a geometric law on $\{0,1,2,\ldots\}$ with mean $b_i$.  Zero-size events can
be removed by thinning; conditional on a nonzero event, the burst is positive
geometric with mean $1+b_i$.  In that convention, the corresponding baseline is
$1+b_i$.

\begin{figure}[t]
\centering
\begin{tikzpicture}[scale=2.0]

\fill[red!15] (0,0) rectangle (1,1);

\draw[->, thick] (0,0) -- (2.7,0)
node[right] {$F_{X_1}/\kappa_1$};
\draw[->, thick] (0,0) -- (0,2.7)
node[above] {$F_{X_2}/\kappa_2$};

\draw[thick, red] (1,0) -- (1,1) -- (0,1);
\draw[dashed, gray] (1,1) -- (1,2.5);
\draw[dashed, gray] (1,1) -- (2.5,1);

\fill[black] (1,1) circle (1pt);
\node[above right, font=\scriptsize] at (1,1) {$(1,1)$};
\node[below, font=\scriptsize] at (1,0) {$1$};
\node[left, font=\scriptsize] at (0,1) {$1$};
\node[align=center, font=\scriptsize] at (0.55,0.5)
{Impossibility\\region};

\end{tikzpicture}
\caption{ 
Universal impossibility region for two-component networks with positive
geometric bursts.  The horizontal and vertical coordinates are
$x=F_{X_1}/\kappa_1$ and $y=F_{X_2}/\kappa_2$, respectively.  The shaded square, apart from the point $(1,1)$, is forbidden by
Theorem~\ref{thm:geometric-main}.  For fixed
$\tau_1,\tau_2$, the theorem excludes the larger half-plane
$(x-1)/\tau_1+(y-1)/\tau_2<0$, whose boundary passes through $(1,1)$ with
slope $-\tau_2/\tau_1$.
}
\label{fig:geometric-impossibility}
\end{figure}
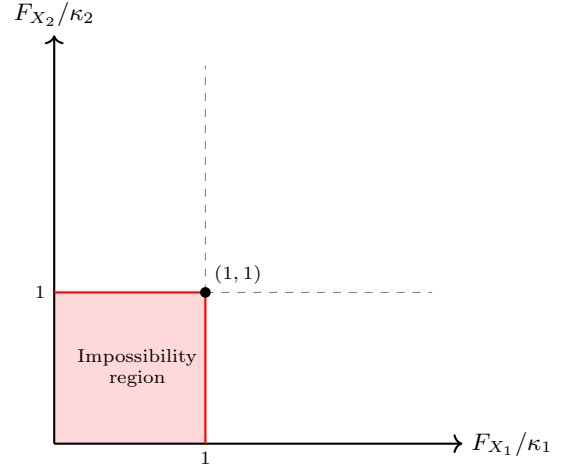

\emph{Proof mechanism.---}
See Section~\ref{sec:S-geometric-proof} of the Supplemental Material for the
complete proof.
The proof extends the stationary information-flow argument for unit births
developed in Refs.~\cite{RKH2025,anderson2026noise}.  The information flows
sum to zero; the death contribution is nonpositive; and a log-sum inequality bounds the birth contribution by an expression
involving the conditional production rates.  The new issue is that one burst may cross several
copy-number thresholds.  The marginal cut-balance identity
\eqref{eq:S-cut} in the Supplemental Material equates the total upward burst
flux crossing a threshold with the downward degradation flux.  For a positive
geometric law,
\begin{equation}
        \mathbb P(K_i=k)
        =
        p_i\,\mathbb P(K_i\ge k).
        \label{eq:geometric-tail-match-main}
\end{equation}
Combining \eqref{eq:geometric-tail-match-main} with the cut-balance identity
\eqref{eq:S-cut} collapses the sum over all thresholds crossed by a burst to a
term involving $\E_\pi[X_i f_i(X_{-i})]$, and hence to the covariance in
\eqref{eq:fano-main}. Summing the resulting componentwise inequalities and using information-flow
cancellation yields \eqref{eq:geom-tradeoff-main}.
The proportionality in \eqref{eq:geometric-tail-match-main} characterizes the
geometric law among positive integer-valued burst distributions. For a nongeometric burst distribution, the same calculation leaves an
additional signed term with no universal sign.  This obstruction is genuine:
as the example above shows, the conclusion can fail.

\emph{A structural impossibility theorem for general bursts.---}
A second impossibility mechanism comes from the regulatory architecture.
Extending the signed-monotonicity result of
Ref.~\cite{anderson2026noise} from unit births to arbitrary positive burst
laws, we show that a globally coherent sign structure forces every component
to remain at or above its own baseline.

Fix a sign vector  
\[
        \sigma=(\sigma_1,\dots,\sigma_N)\in\{-1,1\}^N
\]
and define the associated signed order by
\begin{equation}
        x\preceq_\sigma y
        \quad\Longleftrightarrow\quad
        \sigma_jx_j\le\sigma_jy_j
        \qquad
        \text{for every }j.
        \label{eq:signed-order-main}
\end{equation}
We say that the network is \emph{signed monotone} if there exists such a sign vector for which
\begin{equation}
        x\preceq_\sigma y
        \quad\Longrightarrow\quad
        \sigma_i f_i(x_{-i})
        \le
        \sigma_i f_i(y_{-i})
        \qquad
        \text{for every }i.
        \label{eq:signed-monotonicity-main}
\end{equation}
Thus, after possibly reversing the order on selected coordinates, every
production rate varies in the same order-preserving direction.

The same signed-order framework was developed for unit births in
Ref.~\cite{anderson2026noise}. When each dependence of $f_i$ on a coordinate
$x_j$ is either globally nondecreasing or globally nonincreasing, the
graph-theoretic criterion is especially simple: after edge directions are
ignored, every cycle must contain an even number of negative interactions.
Such signed interaction graphs are called \textit{structurally balanced}.

\begin{theorem}[Termwise bound under signed monotonicity]
\label{thm:signed-main}
Suppose Assumption~\ref{ass:S-standing} holds, each $K_i$ is positive
integer-valued with finite second moment, and the network satisfies
\eqref{eq:signed-monotonicity-main}.  Then, for every
$i\in\{1,\dots,N\}$,
\begin{equation}
        \Cov_\pi\bigl(f_i(X_{-i}),X_i\bigr)\ge0,
        \quad \text{and so} \quad
        F_{X_i}\ge B_i.
        \label{eq:signed-termwise-main}
\end{equation}
\end{theorem}

Unlike Theorem~\ref{thm:geometric-main}, this result places no restriction on
the burst law and gives a separate bound for every component. 
 In particular, whenever each regulatory interaction has a fixed sign and every cycle contains an even
number of negative interactions, every component lies at or above its own
constant-rate baseline.

The proof, given in
Section~\ref{sec:S-signed-monotonicity} of the Supplemental Material, follows
the signed-order argument of Ref.~\cite{anderson2026noise}.
Ordered copies of the burst process are coupled
using the same burst-size realization for production events occurring in both
copies; signed monotonicity ensures that events occurring in only one copy
cannot reverse the order.  The stationary law is then positively associated
with respect to this order, meaning that observables increasing in the signed
order have nonnegative covariance.  Applying this property to
$\sigma_i f_i(X_{-i})$ and $\sigma_iX_i$ gives
\[
        \Cov_\pi\bigl(f_i(X_{-i}),X_i\bigr)\ge0.
\]
Equation~\eqref{eq:fano-main} then gives
\eqref{eq:signed-termwise-main}.  

The counterexample above lies outside this structural class.  In that example,
$f_1$ is increasing whereas $f_2$ is decreasing, so the two-component
feedback cycle contains exactly one negative interaction and therefore
violates the even-negative-cycle condition.  Thus the example lies outside
both classes covered by our positive theorems: its burst law is nongeometric
and its interaction graph is not structurally balanced.

\emph{Discussion.---}
The question considered here is a fundamental one for stochastic interaction
networks: allowing arbitrary numbers of components, arbitrary interaction
topologies, arbitrary burst laws, and arbitrary nonlinear cross-regulatory
rates, can one reduce intrinsic fluctuations without paying for that
reduction elsewhere?  Our results show that the answer depends on the burst law.  We gave an
explicit example showing that the most general version of the conjecture is
\textit{false}.  However, two broad impossibility mechanisms remain.  Most
importantly, geometric bursts enforce the global weighted tradeoff
\eqref{eq:geom-tradeoff-main} for arbitrary topology, establishing the
conjecture in arguably its most important special case.  Further, even for
arbitrary positive burst laws, a globally consistent sign structure on the
interaction graph enforces the stronger termwise bounds
$F_{X_i}\ge B_i$.  Neither result assumes detailed balance, proximity to
equilibrium, or a large-copy-number approximation such as the linear noise
approximation.

Many fundamental questions remain.  First, one would like to identify
nongeometric burst laws for which some universal tradeoff survives, and to
quantify how \eqref{eq:geom-tradeoff-main} changes when the burst law is
perturbed away from geometric.  Second, the instantaneous geometric-burst
model should be compared with a model in which a short-lived intermediate,
such as mRNA, explicitly produces the output molecules before degrading
\cite{thattai2001intrinsic, shahrezaei2008analytical}.  The geometric-burst model arises in the limit as the intermediate lifetime
tends to zero, but the corresponding fluctuation tradeoff for a finite
intermediate lifetime is open.
Finally, it is
important to determine whether any analogous impossibility principle remains
when cross-regulation acts through degradation as well as production.

\subsection{Acknowledgments}

The author used OpenAI's ChatGPT  as an interactive
assistant for bounded tasks involving exposition, algebraic manipulation,
symbolic evaluation of the finite systems arising in the counterexample, and
preparation of figure code.  The author formulated each task, examined the
resulting output, and independently verified every calculation and argument
before incorporating any material into the manuscript.  The mathematical
ideas, analysis, and conclusions, including any typos or errors, are entirely
the author's own.

\bibliographystyle{apsrev4-2}
\bibliography{PRL_bib}

\clearpage
\onecolumngrid
\appendix
\section*{Supplemental Material: Proofs and Details}
\addcontentsline{toc}{section}{Supplemental Material: Proofs and Details}

\setcounter{secnumdepth}{2}

\setcounter{section}{0}
\renewcommand{\thesection}{S\arabic{section}}
\renewcommand{\thesubsection}{\thesection.\arabic{subsection}}

\setcounter{equation}{0}
\renewcommand{\theequation}{S\arabic{equation}}
\setcounter{theorem}{0}
\renewcommand{\thetheorem}{S\arabic{theorem}}
\setcounter{lemma}{0}
\renewcommand{\thelemma}{S\arabic{lemma}}
\setcounter{proposition}{0}
\renewcommand{\theproposition}{S\arabic{proposition}}
\setcounter{corollary}{0}
\renewcommand{\thecorollary}{S\arabic{corollary}}

\setcounter{assumption}{0}
\renewcommand{\theassumption}{S\arabic{assumption}}

\setcounter{definition}{0}
\renewcommand{\thedefinition}{S\arabic{definition}}

\setcounter{remark}{0}
\renewcommand{\theremark}{S\arabic{remark}}

The Supplemental Material is organized into three sections.
Section~\ref{sec:S-geometric-proof} proves the geometric-burst tradeoff,
Theorem~\ref{thm:geometric-main}.  We first state the standing analytical
assumptions and give checkable Foster--Lyapunov conditions under which they
hold.  We then establish the stationary moment, cut-balance, and
information-flow identities for arbitrary positive integer-valued burst
distributions with finite second moments.  In particular, information-flow
cancellation does not require the burst law to be geometric.  The geometric
assumption enters only in the componentwise information-flow estimate, where
the geometric tail identity allows the birth contribution to be evaluated by
a cut balance.  Section~\ref{sec:S-counterexamples} analyzes the family of
nongeometric counterexamples introduced in
\eqref{eq:counter-burst-main}--\eqref{eq:analytic-counter-rates-main}, and
Section~\ref{sec:S-signed-monotonicity} proves the termwise bound under signed
monotonicity, Theorem~\ref{thm:signed-main}.

\section{Proof of the geometric-burst tradeoff}
\label{sec:S-geometric-proof}

\subsection{Standing assumptions and sufficient conditions}

Let the state space of the continuous-time Markov chain be denoted $S:=\mathbb Z_{\ge0}^N$ and let $s(x):=\|x\|_1=\sum_{i=1}^N x_i$.  Also, let
\[
        \Lambda(x)  :=  \sum_{i=1}^N f_i(x_{-i}) +     \sum_{i=1}^N \frac{x_i}{\tau_i}
\]
denote the total exit rate from $x$.  Recall that
\[
        \kappa_i:=\E [K_i].
\]

We work under the following standing hypotheses.  This assumption is analogous to Assumption 2.1 in \cite{anderson2026noise}.

\begin{assumption}[Standing hypotheses]
\label{ass:S-standing}
The following conditions hold.
\begin{enumerate}
\item[\textnormal{(A0)}]
The continuous-time Markov chain $X$ is irreducible on $S$ and admits a stationary distribution
$\pi$.

\item[\textnormal{(A1)}]
The stationary distribution has finite second moments:
\[
        \E_\pi[s(X)^2]<\infty.
\]

\item[\textnormal{(A2)}]
For each $i\in\{1,\dots,N\}$, the stationary mean production rate satisfies
\[
        \mu_i
        :=
        \E_\pi[f_i(X_{-i})]
        \in(0,\infty).
\]

\item[\textnormal{(A3)}]
The total exit rate has at-most-linear growth: there exists
$C_\Lambda<\infty$ such that
\[
        \Lambda(x)
        \le
        C_\Lambda(1+s(x)),
        \qquad \forall \ x\in S.
\]
\end{enumerate}
\end{assumption}

Note that \textnormal{(A0)} implies that $\pi$ has full support.  In particular, all stationary
conditional distributions used below are well defined.  Nonexplosion is
not imposed separately; it follows from
Assumption~\ref{ass:S-standing}\textnormal{(A1)} and
\textnormal{(A3)}. See \cite[Lemma 3.1]{anderson2026noise} for a proof of this fact.

While the conditions of Assumption \ref{ass:S-standing} are the analytical assumptions needed for the theory to go through, it is often more convenient to have simple, checkable conditions on the basic model itself.  The following theorem provides such conditions.  Essentially, it says that so long as the total mean production rate grows at most linearly, with linear growth coefficient strictly smaller than the weakest degradation rate, then the necessary moment conditions hold.   Importantly, an obvious corollary is that if the $f_i$ are uniformly bounded above and uniformly bounded away from zero, then Assumption~\ref{ass:S-standing} holds.  Finally, we note that the theorem below applies to general burst distributions, not only geometric ones. 

\begin{theorem}[Sufficient conditions for Assumption~\ref{ass:S-standing}]
\label{thm:S-sufficient-conditions}
Suppose that each $K_i$ has support on the positive integers and satisfies
\[
        \kappa_i=\E[K_i]<\infty,  \qquad \E[K_i^2]<\infty.
\]
Let $\tau_{\max}:=\max_{1\le i\le N}\tau_i$ and assume that there exist constants
\[
        \varepsilon>0,
        \qquad
        A<\infty,
        \qquad
        0\le B<\frac{1}{\tau_{\max}},
\]
such that
\begin{equation}
        f_i(v)\ge\varepsilon
        \qquad
        \text{for every \(i\) and every \(v\in\mathbb Z_{\ge0}^{N-1}\),}
        \label{eq:S-positive-rates}
\end{equation}
and
\begin{equation}
        \sum_{i=1}^N
        \kappa_i f_i(x_{-i})
        \le
        A+B s(x),
        \qquad \forall x\in S.
        \label{eq:S-FL-condition}
\end{equation}
Then the chain is irreducible, nonexplosive, and positive recurrent.  Moreover, its
unique stationary distribution \(\pi\) satisfies
Assumption~\ref{ass:S-standing}.
\end{theorem}

\begin{proof}
The proof is essentially the same as that of
\cite[Theorem 2.3]{anderson2026noise}, with one modification to the
Foster--Lyapunov calculation.  In the unit-birth model, a production event
increases $s(x)=\|x\|_1$ by one, whereas here a production event in
coordinate $i$ increases $s(x)$ by $K_i$.  Thus, for
$V(x)=(1+s(x))^2$,
\[
\E\!\left[V(x+K_i e_i)-V(x)\right]
=
2(1+s(x))\kappa_i+\E[K_i^2].
\]
Consequently, with
\[
C_K:=\max_i\frac{\E[K_i^2]}{\kappa_i}<\infty,
\]
the drift satisfies
\[
\begin{aligned}
(LV)(x)
&\le
\left(2(1+s(x))+C_K\right)
\sum_{i=1}^N\kappa_i f_i(x_{-i})
-
(2s(x)+1)\sum_{i=1}^N\frac{x_i}{\tau_i} \\
&\le
\left(2(1+s(x))+C_K\right)(A+Bs(x))
-
\frac{(2s(x)+1)s(x)}{\tau_{\max}}.
\end{aligned}
\]
The leading quadratic coefficient is
\[
2B-\frac{2}{\tau_{\max}}<0.
\]
The remainder of the Foster--Lyapunov argument from
\cite[Theorem 2.3]{anderson2026noise} therefore applies without further
change and yields  nonexplosion, positive recurrence, and
finite stationary second moments.  The verification of irreducibility is straightforward and the proof of
Assumption~\ref{ass:S-standing}\textnormal{(A2)}--\textnormal{(A3)} is
also identical, using $\kappa_i\ge1$ and
\eqref{eq:S-FL-condition}.
\end{proof}

\subsection{Stationary generator and logarithmic integrability}

We  recall two general results from
\cite[Lemmas 3.2--3.4]{anderson2026noise}.  Their proofs do not use the
unit-birth structure and therefore apply without change to the present
burst model.

\begin{lemma}[Stationary generator identity]
Let $X$ be an arbitrary continuous-time Markov chain on a countable state
space $S$, with transition rates $q(x,y)$ and infinitesimal generator $L$, and suppose
that $X$ admits a full-support stationary distribution $\pi$ satisfying
$\E_\pi[\Lambda(X)]<\infty$.
\label{lem:stationary_identity_lemma}
Let $\varphi:S\to\mathbb R$ satisfy
\begin{equation}
        \sum_{x\in S}\pi(x)
        \sum_{y\ne x}q(x,y)
        |\varphi(y)-\varphi(x)|
        <\infty.
        \label{eq:S-domain}
\end{equation}
Then, for every $x\in S$, the series defining
\[
        (L\varphi)(x)
        =
        \sum_{y\ne x}q(x,y)
        [\varphi(y)-\varphi(x)]
\]
is absolutely convergent.  Moreover, $L\varphi\in L^1(\pi)$ and
\[
        \E_\pi[(L\varphi)(X)]=0.
\]
\end{lemma}

\begin{proof}
This is exactly \cite[Lemma 3.2]{anderson2026noise}.
\end{proof}

\begin{lemma}[Logarithmic integrability]
\label{lem:S-log-integrability}
Let $\pi$ be any full-support probability distribution on
$S=\mathbb Z_{\ge0}^N$ satisfying $\E_\pi[s(X)^2]<\infty$.
If, in addition, $\pi$ is stationary for a continuous-time Markov chain
whose total exit rate satisfies Assumption~\ref{ass:S-standing}\textnormal{(A3)},
then
\begin{equation}
        \sum_{x\in S}
        \pi(x)(1+s(x))|\log\pi(x)|
        <\infty,
        \label{eq:S-weighted-entropy}
\end{equation}
and
\begin{equation}
        \sum_{x\in S}\pi(x)
        \sum_{y\ne x}q(x,y)
        \bigl(
        |\log\pi(x)|+|\log\pi(y)|
        \bigr)
        <\infty.
        \label{eq:S-edge-entropy}
\end{equation}
In particular, for $g(x):=-\log\pi(x)$, we have
$Lg\in L^1(\pi)$ and
\[
        \E_\pi[(Lg)(X)]=0.
\]
\end{lemma}

\begin{proof}
This follows directly from
\cite[Lemmas 3.3 and 3.4]{anderson2026noise}.  Those proofs require only
$\E_\pi[s(X)^2]<\infty$ and the at-most-linear growth of $\Lambda$, and
not the unit-birth form of the transitions.
\end{proof}

\subsection{Moment and cut-balance identities}

Throughout this subsection, the burst distributions are arbitrary positive
integer-valued distributions with finite second moments; they need not be
geometric.  We assume throughout that Assumption~\ref{ass:S-standing} holds.

Fix $i\in\{1,\dots,N\}$.  When separating the $i$th coordinate from the
others, we write $x=(n,v)$, where $n=x_i$ and $v=x_{-i}$, and, following the notation of \cite{anderson2026noise}, define
\begin{align*}
        \pi_i(n)
        &:=\sum_{v\in \Z^{N-1}_{\ge 0}} \pi(n,v)
        =\mathbb P_\pi(X_i=n),\\
        \pi_n^i(v)
        &:=\frac{\pi(n,v)}{\pi_i(n)}
        =\mathbb P_\pi(X_{-i}=v\mid X_i=n),\\
        m_i(n)
        &:= \sum_{v\in \Z^{N-1}_{\ge 0}} \pi_n^i(v)f_i(v)
        =\E_\pi[f_i(X_{-i})\mid X_i=n].
\end{align*}
These conditional distributions are well defined because $\pi$ has full
support.

The following result is the burst analogue of \cite[Theorem 2.2, part (1)]{anderson2026noise}.  The proof is essentially the same, but with the first two moments of $K_i$ playing a role.

\begin{lemma}[Stationary moment identities]
\label{lem:S-moment-identities}
For every $i\in\{1,\dots,N\}$,
\begin{equation}
        \frac{\E_\pi[X_i]}{\tau_i}
        =
        \kappa_i\mu_i
        =
        \kappa_i\E_\pi[f_i(X_{-i})].
        \label{eq:S-mean}
\end{equation}
Moreover, defining
\begin{equation}
        B_i
        :=
        \frac{\E[K_i^2]+\E[K_i]}{2\E[K_i]}
        =
        \frac{\E[K_i^2]+\kappa_i}{2\kappa_i},
        \label{eq:S-baseline}
\end{equation}
we have
\begin{equation}
        F_{X_i}-B_i
        =
        \frac{\Cov_\pi(f_i(X_{-i}),X_i)}{\mu_i}.
        \label{eq:S-fano}
\end{equation}
\end{lemma}

\begin{proof}
First take $\varphi(x)=x_i$.  To verify the integrability condition
\eqref{eq:S-domain} in Lemma~\ref{lem:stationary_identity_lemma}, note that
\[
\sum_{y\ne x}q(x,y)|\varphi(y)-\varphi(x)|
=
\kappa_i f_i(x_{-i})+\frac{x_i}{\tau_i}.
\]
Its stationary expectation is finite by
Assumption~\ref{ass:S-standing}\textnormal{(A1)}--\textnormal{(A2)}.
Therefore Lemma~\ref{lem:stationary_identity_lemma} applies, and
\[
0
=
\E_\pi[(L\varphi)(X)]
=
\kappa_i\E_\pi[f_i(X_{-i})]
-
\frac{\E_\pi[X_i]}{\tau_i},
\]
which proves \eqref{eq:S-mean}.

Next take $\varphi(x)=x_i^2$.  Only jumps in coordinate $i$ contribute.  A burst of size $k$ changes $\varphi$ by
\[
        (x_i+k)^2-x_i^2=2kx_i+k^2,
\]
while a death changes $\varphi$ by
\[
        (x_i-1)^2-x_i^2=-2x_i+1.
\]
Therefore,
\begin{align*}
(L\varphi)(x)
&=
f_i(x_{-i})
\sum_{k\ge1}p_{i,k}(2kx_i+k^2)
+
\frac{x_i}{\tau_i}(-2x_i+1) \\
&=
f_i(x_{-i})
\left(
2\kappa_i x_i+\E[K_i^2]
\right)
+
\frac{x_i}{\tau_i}(-2x_i+1).
\end{align*}
Moreover,
\begin{align*}
\sum_{y\ne x}q(x,y)|\varphi(y)-\varphi(x)|
&=
f_i(x_{-i})\sum_{k\ge1}p_{i,k}(2kx_i+k^2)
+
\frac{x_i}{\tau_i}|-2x_i+1| \\
&\le
2\kappa_i x_i f_i(x_{-i})
+\E[K_i^2]f_i(x_{-i})
+\frac{2x_i^2+x_i}{\tau_i}.
\end{align*}
Since $f_i(x_{-i})\le\Lambda(x)\le C_\Lambda(1+s(x))$ and
$x_i\le s(x)$, the stationary expectation of the right-hand side is finite by
Assumption~\ref{ass:S-standing}\textnormal{(A1)}--\textnormal{(A3)}.
Hence Lemma~\ref{lem:stationary_identity_lemma} applies and gives
\begin{align}
0
&=
\E_\pi[(L\varphi)(X)] \notag\\
&=
2\kappa_i\E_\pi[f_i(X_{-i})X_i]
+\E[K_i^2]\mu_i
-\frac{2}{\tau_i}\E_\pi[X_i^2]
+\frac{1}{\tau_i}\E_\pi[X_i].
\label{eq:S-second-moment-balance}
\end{align}
Using \eqref{eq:S-mean} and rearranging,
\begin{equation}
        \Var_\pi(X_i)
        =
        B_i\E_\pi[X_i]
        +
        \kappa_i\tau_i
        \Cov_\pi(f_i(X_{-i}),X_i).
\end{equation}
Finally, divide by $\E_\pi[X_i]=\kappa_i\tau_i\mu_i$
to obtain \eqref{eq:S-fano}.
\end{proof}

We next establish the burst analogue of the adjacent-level current and
marginal flux balance used in
\cite[Eq.~(2.9) and Lemma 3.6]{anderson2026noise}.  Define
\begin{equation*}
        T_{i,r}:=\mathbb P(K_i\ge r),
        \qquad r\ge1.
\end{equation*}
For $j\ge1$ and $v\in\mathbb Z_{\ge0}^{N-1}$, define the stationary
fiber current across the cut separating the levels $X_i<j$ and
$X_i\ge j$ by
\begin{equation}
\mathcal J_i(j;v)
:=
\sum_{n=0}^{j-1}
\pi(n,v)f_i(v)T_{i,j-n}
-
\pi(j,v)\frac{j}{\tau_i}.
\label{eq:S-cut-current}
\end{equation}
The first term is the total upward burst flux across the cut within the
fiber $v$, while the second is the downward death flux.  If
$K_i\equiv1$, only the term $n=j-1$ remains, and
\eqref{eq:S-cut-current} reduces exactly to the adjacent-edge current
used in the unit-birth paper \cite{anderson2026noise}.

\begin{lemma}[Marginal cut balance]
\label{lem:S-cut-balance}
For every $i\in\{1,\dots,N\}$ and every $j\ge1$,
\begin{equation}
        \sum_v\mathcal J_i(j;v)=0.
        \label{eq:S-zero-cut-current}
\end{equation}
Equivalently,
\begin{equation}
        \frac{j\pi_i(j)}{\tau_i}
        =
        \sum_{n=0}^{j-1}
        \pi_i(n)m_i(n)T_{i,j-n}.
        \label{eq:S-cut}
\end{equation}
\end{lemma}

\begin{proof}
This is the same argument used to prove
\cite[Lemma 3.6]{anderson2026noise}.  Fix $i$ and $j\ge1$, and let
\[
        \varphi_j(x):=\mathbf 1_{\{x_i\ge j\}}.
\]
Since $\varphi_j$ is bounded,
\[
\sum_x\pi(x)\sum_{y\ne x}
q(x,y)|\varphi_j(y)-\varphi_j(x)|
\le
\E_\pi[\Lambda(X)]
<\infty.
\]
Therefore Lemma~\ref{lem:stationary_identity_lemma} gives
$\E_\pi[L\varphi_j]=0$.

For $\varphi_j(x)=\mathbf 1_{\{x_i\ge j\}}$, a death changes $\varphi_j$
only when $x_i=j$, while a burst starting from level $n<j$ changes
$\varphi_j$ precisely when $K_i\ge j-n$.  Therefore
\[
(L\varphi_j)(n,v)
=
f_i(v)\,T_{i,j-n}\mathbf 1_{\{n<j\}}
-
\frac{j}{\tau_i}\mathbf 1_{\{n=j\}}.
\]
Using $\E_\pi[L\varphi_j]=0$, we obtain
\begin{align*}
0
&=
\sum_{n=0}^{j-1}\sum_v
\pi(n,v)f_i(v)T_{i,j-n}
-
\frac{j}{\tau_i}\sum_v\pi(j,v) \\
&=
\sum_{n=0}^{j-1}
\pi_i(n)m_i(n)T_{i,j-n}
-
\frac{j\pi_i(j)}{\tau_i}.
\end{align*}
Rearranging proves \eqref{eq:S-cut}.  Equation~\eqref{eq:S-zero-cut-current}
then follows immediately from the definition of $\mathcal J_i(j;v)$.
\end{proof}

\subsection{Information flow and cancellation}

Throughout this subsection, the burst distributions are arbitrary positive
integer-valued distributions with finite second moments; they need not be
geometric.  We assume throughout that Assumption~\ref{ass:S-standing} holds.

Recall that $L=\sum_{i=1}^N L_i$, where
\begin{align}
(L_i\varphi)(x)
&:=
f_i(x_{-i})\sum_{k\ge1}p_{i,k}
\bigl[\varphi(x+ke_i)-\varphi(x)\bigr]
+ \frac{x_i}{\tau_i}
\bigl[\varphi(x-e_i)-\varphi(x)\bigr],
\label{eq:S-coordinate-generator}
\end{align}
with the death term omitted when $x_i=0$.

Define the logarithmic potentials
\begin{equation}
        g(x):=-\log\pi(x),
        \qquad
        g_i(x):=-\log\pi_i(x_i),
        \qquad
        h_i(x):=-\log\pi_{x_i}^i(x_{-i}).
        \label{eq:S-log-potentials}
\end{equation}
Since
\[
        \pi(x)=\pi_i(x_i)\pi_{x_i}^i(x_{-i}),
\]
we have
\begin{equation}
        g(x)=g_i(x)+h_i(x).
        \label{eq:S-log-decomposition}
\end{equation}

The following definition is the burst analogue of
\cite[Definition 2.1]{anderson2026noise}.

\begin{definition}[Information flow associated with coordinate $i$]
\label{def:S-information-flow}
Assuming that the series below is absolutely convergent, define
\begin{align}
(\dot I)_{X_i}
&:=
\sum_{n\ge0}
\sum_{v\in\mathbb Z_{\ge0}^{N-1}}
\sum_{k\ge1}
\pi(n,v)f_i(v)p_{i,k}
\log\frac{\pi_{n+k}^i(v)}{\pi_n^i(v)}
\notag\\
&\quad+
\sum_{n\ge1}
\sum_{v\in\mathbb Z_{\ge0}^{N-1}}
\pi(n,v)\frac{n}{\tau_i}
\log\frac{\pi_{n-1}^i(v)}{\pi_n^i(v)}.
\label{eq:S-flow-expanded}
\end{align}
\end{definition}

The following result is the burst analogue of
\cite[Lemma 3.5]{anderson2026noise}.  The proof is unchanged in
structure: one uses the comparison $0\le h_i\le g$, restricts the
edge-integrability estimate to transitions in coordinate $i$, and then
identifies the resulting series with $-\E_\pi[(L_i h_i)(X)]$.  The only
difference is that the single unit-birth transition
$(n,v)\to(n+1,v)$ is replaced by the family of burst transitions
$(n,v)\to(n+k,v)$, summed over $k\ge1$.  We include the details because
these burst transitions no longer pair one-to-one with the adjacent
death transitions, so the information flow is written as separate birth
and death contributions.

\begin{lemma}[Absolute convergence and generator representation]
\label{lem:S-flow-representation}
For every $i\in\{1,\dots,N\}$, the series in
\eqref{eq:S-flow-expanded} is absolutely convergent.  Moreover,
$L_i h_i\in L^1(\pi)$ and
\begin{equation}
        (\dot I)_{X_i}
        =
        -\E_\pi[(L_i h_i)(X)].
        \label{eq:S-flow-generator}
\end{equation}
\end{lemma}

\begin{proof}
Since $\pi_i(n)\le1$ and
\[
        \pi_n^i(v)=\frac{\pi(n,v)}{\pi_i(n)}\ge\pi(n,v),
\]
we have
\begin{equation}
        0\le g_i(n,v)\le g(n,v),
        \qquad
        0\le h_i(n,v)\le g(n,v).
        \label{eq:S-log-comparison}
\end{equation}

Consider first the birth transitions in coordinate $i$.  By
\eqref{eq:S-log-comparison},
\begin{align*}
&\sum_{n,v,k}
\pi(n,v)f_i(v)p_{i,k}
\left|h_i(n+k,v)-h_i(n,v)\right|
\\
&\qquad\le
\sum_{n,v,k}
\pi(n,v)f_i(v)p_{i,k}
\bigl[g(n+k,v)+g(n,v)\bigr]
<\infty,
\end{align*}
where finiteness follows by restricting the edge-integrability estimate
\eqref{eq:S-edge-entropy} to the $i$th-coordinate birth transitions.

Similarly, for the death transitions,
\begin{align*}
&\sum_{n\ge1,v}
\pi(n,v)\frac{n}{\tau_i}
\left|h_i(n-1,v)-h_i(n,v)\right|
\\
&\qquad\le
\sum_{n\ge1,v}
\pi(n,v)\frac{n}{\tau_i}
\bigl[g(n-1,v)+g(n,v)\bigr]
<\infty.
\end{align*}
Thus the series in \eqref{eq:S-flow-expanded} is absolutely convergent,
and $L_i h_i\in L^1(\pi)$.

The following termwise calculation is therefore justified:
\begin{align*}
-\E_\pi[(L_i h_i)(X)]
&=
\sum_{n,v,k}
\pi(n,v)f_i(v)p_{i,k}
\bigl[h_i(n,v)-h_i(n+k,v)\bigr]
\\
&\quad+
\sum_{n\ge1,v}
\pi(n,v)\frac{n}{\tau_i}
\bigl[h_i(n,v)-h_i(n-1,v)\bigr].
\end{align*}
Since $h_i(n,v)=-\log\pi_n^i(v)$, this becomes exactly
\eqref{eq:S-flow-expanded}, proving \eqref{eq:S-flow-generator}.
\end{proof}

To prove information-flow cancellation, it remains to show that
\[
        \E_\pi[(L_i g_i)(X)]=0.
\]
In the unit-birth argument, this follows by pairing adjacent marginal levels
using the corresponding flux-balance identity.  For general bursts, a single
production event may skip several levels, so that adjacent-level pairing is
no longer available.  Instead, we average the coordinate-$i$ production rate
conditional on $X_i=n$.  This produces a one-dimensional burst generator with
birth rate $m_i(n)$ and stationary distribution $\pi_i$.  Applying the
stationary generator identity to $-\log\pi_i$ then gives the required
equality.

\begin{lemma}[Stationarity of the conditional-rate marginal generator]
\label{lem:S-marginal-generator}
For each $i$, define
\begin{equation}
(\overline L_i\psi)(n)
:=
m_i(n)\sum_{k\ge1}p_{i,k}
\bigl[\psi(n+k)-\psi(n)\bigr]
+
\frac{n}{\tau_i}
\bigl[\psi(n-1)-\psi(n)\bigr],
\label{eq:S-marginal-generator}
\end{equation}
with the death term omitted when $n=0$.  Then $\pi_i$ is stationary for
$\overline L_i$.  Moreover,
\begin{equation}
        \E_\pi[(L_i g_i)(X)]=0.
        \label{eq:S-marginal-log-zero}
\end{equation}
\end{lemma}

\begin{proof}
First note that $m_i(n)<\infty$ for every $n$.  Indeed,
\[
        \sum_{n\ge0}\pi_i(n)m_i(n)=\mu_i<\infty,
\]
and $\pi_i(n)>0$ by full support.

Let $\psi:\mathbb Z_{\ge0}\to\mathbb R$ be bounded, and consider the
full-state test function $\varphi(x)=\psi(x_i)$.  Its integrability
condition in Lemma~\ref{lem:stationary_identity_lemma} is bounded by
\[
        2\|\psi\|_\infty
        \E_\pi\left[
        f_i(X_{-i})+\frac{X_i}{\tau_i}
        \right]
        <\infty.
\]
Hence
\[
        0
        =
        \E_\pi[(L\varphi)(X)]
        =
        \E_\pi[(L_i\varphi)(X)],
\]
because jumps in coordinates other than $i$ do not change $\varphi$.
Conditioning on $X_i=n$ gives
\[
        0
        =
        \sum_{n\ge0}
        \pi_i(n)(\overline L_i\psi)(n).
\]
Thus $\pi_i$ is stationary for $\overline L_i$.

Now let
\[
        \widetilde g_i(n):=-\log\pi_i(n),
\]
so that $g_i(x)=\widetilde g_i(x_i)$.  We verify the integrability
condition needed to apply Lemma~\ref{lem:stationary_identity_lemma} to
$\widetilde g_i$ and $\overline L_i$.  We have
\begin{align*}
\sum_{n\ge0}\pi_i(n)m_i(n) &
\sum_{k\ge1}p_{i,k}
\left|\widetilde g_i(n+k)-\widetilde g_i(n)\right| +
\sum_{n\ge1}\pi_i(n)\frac{n}{\tau_i}
\left|\widetilde g_i(n-1)-\widetilde g_i(n)\right|
\\
&=
\sum_{n,v,k}
\pi(n,v)f_i(v)p_{i,k}
\left|g_i(n+k,v)-g_i(n,v)\right|
 +
\sum_{n\ge1,v}
\pi(n,v)\frac{n}{\tau_i}
\left|g_i(n-1,v)-g_i(n,v)\right|.
\end{align*}
By \eqref{eq:S-log-comparison}, this is bounded by the corresponding
portion of \eqref{eq:S-edge-entropy}, and is therefore finite.

Also,
\[
        \sum_{n\ge0}\pi_i(n)
        \left(m_i(n)+\frac{n}{\tau_i}\right)
        =
        \mu_i+\frac{\E_\pi[X_i]}{\tau_i}
        <\infty.
\]
Lemma~\ref{lem:stationary_identity_lemma}, applied to the one-dimensional
generator $\overline L_i$, therefore gives
\[
        \sum_{n\ge0}
        \pi_i(n)(\overline L_i\widetilde g_i)(n)
        =
        0.
\]
Finally, conditioning on $X_i$ shows that
\[
        \E_\pi[(L_i g_i)(X)]
        =
        \sum_{n\ge0}
        \pi_i(n)(\overline L_i\widetilde g_i)(n),
\]
which proves \eqref{eq:S-marginal-log-zero}.
\end{proof}

We can now prove the cancellation identity.  The argument is the burst
analogue of the entropy-decomposition argument in
\cite[Theorem 2.2, part (3)]{anderson2026noise}.

\begin{lemma}[Information-flow cancellation]
\label{lem:S-flow-cancellation}
We have
\begin{equation}
        \sum_{i=1}^N(\dot I)_{X_i}=0.
        \label{eq:S-flow-cancellation}
\end{equation}
\end{lemma}

\begin{proof}
By Lemma~\ref{lem:S-log-integrability},
\[
        Lg\in L^1(\pi),
        \qquad \text{and} \qquad 
        \E_\pi[(Lg)(X)]=0.
\]
For each $i$, \eqref{eq:S-log-decomposition} gives
\[
        L_i g=L_i(g_i+h_i)=L_i g_i+L_i h_i.
\]
Since $L=\sum_{i=1}^N L_i$, it follows pointwise that
\[
\begin{aligned}
        Lg
        &=
        \sum_{i=1}^N L_i g =
        \sum_{i=1}^N L_i(g_i+h_i) =
        \sum_{i=1}^N L_i g_i
        +
        \sum_{i=1}^N L_i h_i.
\end{aligned}
\]
The terms $L_i g_i$ and $L_i h_i$ are integrable by
Lemmas~\ref{lem:S-marginal-generator} and
\ref{lem:S-flow-representation}.  Therefore, taking stationary
expectations gives
\[
\begin{aligned}
0
&=
\E_\pi[(Lg)(X)] \\
&=
\sum_{i=1}^N\E_\pi[(L_i g_i)(X)]
+
\sum_{i=1}^N\E_\pi[(L_i h_i)(X)].
\end{aligned}
\]
The first sum vanishes by
Lemma~\ref{lem:S-marginal-generator}.  Hence
\[
        \sum_{i=1}^N\E_\pi[(L_i h_i)(X)]=0.
\]
Using the representation
\[
        (\dot I)_{X_i}=-\E_\pi[(L_i h_i)(X)]
\]
from Lemma~\ref{lem:S-flow-representation} gives the result,
\[
        \sum_{i=1}^N(\dot I)_{X_i}=0.
\]
\end{proof}

\subsection{Geometric componentwise inequality and proof of
Theorem~\ref{thm:geometric-main}}

The preceding results apply to arbitrary positive integer-valued burst
distributions with finite second moments.  We now specialize to positive
geometric bursts:
\begin{equation}
        p_{i,k}
        :=
        \mathbb P(K_i=k)
        =
        p_iq_i^{k-1},
        \qquad
        k\ge1,
        \qquad
        q_i:=1-p_i,
        \qquad
        \kappa_i=\frac1{p_i},
        \label{eq:S-geometric-law}
\end{equation}
where $p_i\in(0,1]$.  The endpoint $p_i=1$ is the unit-birth case.  We use the convention $q_i^0=1$.  Thus, when $p_i=1$, one has
$q_i=0$, $p_{i,1}=1$, and $p_{i,k}=0$ for every $k\ge2$.

We also recall the notation
\[
        T_{i,r}:=\mathbb P(K_i\ge r).
\]

For the law \eqref{eq:S-geometric-law},
\[
        \E[K_i]=\frac1{p_i},
        \qquad
        \E[K_i^2]=\frac{2-p_i}{p_i^2}.
\]
Consequently, the general burst baseline satisfies
\begin{equation}
        B_i
        =
        \frac{\E[K_i^2]+\E[K_i]}{2\E[K_i]}
        =
        \frac1{p_i}
        =
        \kappa_i.
        \label{eq:S-geometric-baseline}
\end{equation}

The proof below parallels
\cite[Theorem 2.2, part (2)]{anderson2026noise} through the log-sum
and elementary logarithmic estimates.  The new step is that births may
cross several marginal levels.  For geometric bursts,
\[
        q_i^{k-1}=T_{i,k},
\]
and this tail identity allows the resulting nonlocal expression to be
evaluated using the cut-balance identity \eqref{eq:S-cut}.

\begin{proposition}[Geometric componentwise information inequality]
\label{prop:S-geometric-local}
Suppose Assumption~\ref{ass:S-standing} holds and that each $K_i$ satisfies
\eqref{eq:S-geometric-law} for some $p_i \in (0,1]$. Then, for each $i\in\{1,\dots,N\}$,
\begin{equation}
        (\dot I)_{X_i}
        \le
        \frac{1}{\tau_i\kappa_i}
        \bigl(F_{X_i}-\kappa_i\bigr).
        \label{eq:S-geometric-local}
\end{equation}
\end{proposition}

\begin{proof}
We fix $i\in\{1,\dots,N\}$ throughout the proof.  Decompose the
information flow in \eqref{eq:S-flow-expanded} as
\begin{equation}
        (\dot I)_{X_i}
        =
        I_{i,+}+I_{i,-},
        \label{eq:S-flow-plus-minus}
\end{equation}
where
\begin{align}
I_{i,+}
&:=
\sum_{n\ge0}
\sum_{v\in\mathbb Z_{\ge0}^{N-1}}
\sum_{k\ge1}
\pi(n,v)f_i(v)p_iq_i^{k-1}
\log\frac{\pi_{n+k}^i(v)}{\pi_n^i(v)},
\label{eq:S-I-plus}
\\
I_{i,-}
&:=
\sum_{n\ge1}
\sum_{v\in\mathbb Z_{\ge0}^{N-1}}
\pi(n,v)\frac{n}{\tau_i}
\log\frac{\pi_{n-1}^i(v)}{\pi_n^i(v)}.
\label{eq:S-I-minus}
\end{align}
Both series are absolutely convergent by
Lemma~\ref{lem:S-flow-representation}.

We first bound the death contribution.  Recalling that
$\pi(n,v)=\pi_i(n)\pi_n^i(v)$, we have
\begin{align}
I_{i,-}
&=
\sum_{n\ge1}
\frac{n\pi_i(n)}{\tau_i}
\sum_v
\pi_n^i(v)
\log\frac{\pi_{n-1}^i(v)}{\pi_n^i(v)}
\notag\\
&=
-\sum_{n\ge1}
\frac{n\pi_i(n)}{\tau_i}
\KL\bigl(\pi_n^i\|\pi_{n-1}^i\bigr)
\le0,
\label{eq:S-death-nonpositive}
\end{align}
where
\[
        \KL(\alpha\|\beta)
        :=
        \sum_v\alpha(v)\log\frac{\alpha(v)}{\beta(v)}
\]
is the Kullback--Leibler divergence.  This is the direct analogue of
the estimate on the death contribution in the unit-birth proof of \cite{RKH2025,anderson2026noise}.

We next bound $I_{i,+}$.  The log-sum estimates below will produce terms of
the form
\[
        \log\frac{m_i(n+k)}{m_i(n)},
\]
so we first verify that these quantities are well defined.  We claim that for every
$n\ge0$,
\[
        0<m_i(n)<\infty.
\]
Finiteness follows from finiteness of  $\mu_i$,
\[
        \sum_{n\ge0}\pi_i(n)m_i(n) = \mu_i < \infty,
\]
and the fact that each term in the sum is nonnegative and $\pi_i(n)>0$.  Positivity follows
because $\mu_i>0$, so there exists $v_*$ with $f_i(v_*)>0$, while full
support gives $\pi_n^i(v_*)>0$ for every $n$.  Hence
\[
        m_i(n)
        =
        \sum_v \pi_n^i(v)f_i(v)
        >0.
\]

Fix $n\ge0$ and $k\ge1$.  The contribution to $I_{i,+}$ associated
with births from level $n$ to level $n+k$ is
\[
\sum_v
\pi(n,v)f_i(v)p_iq_i^{k-1}
\log\frac{\pi_{n+k}^i(v)}{\pi_n^i(v)}.
\]
If $p_i=1$, all terms with $k\ge2$ vanish.  Thus, in forming the ratios
below, we restrict to those $k$ for which $p_iq_i^{k-1}>0$; in the
endpoint case this means only $k=1$.

Using $\pi(n,v)=\pi_i(n)\pi_n^i(v)$, define
\[
\begin{aligned}
        b_v
        &:=
        \pi_i(n)\pi_n^i(v)f_i(v)p_iq_i^{k-1},\\
        a_v
        &:=
        \pi_i(n)\pi_{n+k}^i(v)f_i(v)p_iq_i^{k-1}.
\end{aligned}
\]
Terms with $f_i(v)=0$ contribute zero and may be omitted.  For the
remaining terms,
\[
        \frac{a_v}{b_v}
        =
        \frac{\pi_{n+k}^i(v)}{\pi_n^i(v)},
\]
so the preceding contribution can be written as
\[
        \sum_v b_v\log\frac{a_v}{b_v}.
\]
Moreover,
\[
        \sum_v b_v
        =
        \pi_i(n)m_i(n)p_iq_i^{k-1},
        \qquad
        \sum_v a_v
        =
        \pi_i(n)m_i(n+k)p_iq_i^{k-1}.
\]
The log-sum inequality,
\[
        \sum_v b_v\log\frac{a_v}{b_v}
        \le
        \left(\sum_v b_v\right)
        \log\frac{\sum_v a_v}{\sum_v b_v},
\]
therefore gives
\begin{align}
&\sum_v
\pi(n,v)f_i(v)p_iq_i^{k-1}
\log\frac{\pi_{n+k}^i(v)}{\pi_n^i(v)} =
\sum_v b_v\log\frac{a_v}{b_v} \le
\left(\sum_v b_v\right)
\log\frac{\sum_v a_v}{\sum_v b_v}
\notag\\
&\qquad=
\pi_i(n)m_i(n)p_iq_i^{k-1}
\log
\frac{
\pi_i(n)m_i(n+k)p_iq_i^{k-1}
}{
\pi_i(n)m_i(n)p_iq_i^{k-1}
}
\notag\\
&\qquad=
\pi_i(n)m_i(n)p_iq_i^{k-1}
\log\frac{m_i(n+k)}{m_i(n)}.
\label{eq:S-logsum-fixed-nk}
\end{align}

We next use the elementary inequality
\begin{equation}
        x\log\frac{y}{x}
        \le
        \frac{xy}{\mu_i}+\mu_i-2x,
        \qquad x,y>0,
        \label{eq:S-xy-ineq}
\end{equation}
which follows by writing $x\log\frac{y}{x} = x \log\frac{y}{\mu_i}+ x\log\frac{\mu_i}{x}$ and applying the basic inequality $\log r\le r-1$ to each term.  Applying
\eqref{eq:S-xy-ineq} to \eqref{eq:S-logsum-fixed-nk} with
$x=m_i(n)$ and $y=m_i(n+k)$ gives
\begin{align}
&\sum_v
\pi(n,v)f_i(v)p_iq_i^{k-1}
\log\frac{\pi_{n+k}^i(v)}{\pi_n^i(v)}
\notag\\
&\qquad\le
\pi_i(n)p_iq_i^{k-1}
\left[
\frac{m_i(n)m_i(n+k)}{\mu_i}
+\mu_i-2m_i(n)
\right].
\label{eq:S-logsum-xy-bound}
\end{align}

Summing \eqref{eq:S-logsum-xy-bound} over $n\ge0$ and $k\ge1$, the
last two terms on the right simplify immediately:
\begin{align*}
\sum_{n\ge0}\sum_{k\ge1}
\pi_i(n)p_iq_i^{k-1}\mu_i
&=
\mu_i
\left(\sum_{n\ge0}\pi_i(n)\right)
\left(\sum_{k\ge1}p_iq_i^{k-1}\right)
=
\mu_i,
\\
\sum_{n\ge0}\sum_{k\ge1}
2\pi_i(n)p_iq_i^{k-1}m_i(n)
&=
2
\left(\sum_{n\ge0}\pi_i(n)m_i(n)\right)
\left(\sum_{k\ge1}p_iq_i^{k-1}\right)
=
2\mu_i.
\end{align*}
Hence
\begin{equation}
I_{i,+}
\le
\frac{1}{\mu_i}
\sum_{n\ge0}\sum_{k\ge1}
\pi_i(n)p_iq_i^{k-1}m_i(n)m_i(n+k)
-\mu_i.
\label{eq:S-I-plus-pre-S}
\end{equation}
We therefore define
\begin{equation}
S_i
:=
\sum_{n\ge0}\sum_{k\ge1}
\pi_i(n)m_i(n)p_iq_i^{k-1}m_i(n+k),
\label{eq:S-bilinear-geometric}
\end{equation}
so that
\begin{equation}
        I_{i,+}\le \frac{S_i}{\mu_i}-\mu_i.
\end{equation}
We now study $S_i$, and prove that it is finite (note that all terms are non-negative, so there is no issue with conditional convergence).

Up to this point, the argument has used only the same log-sum estimates
as in the unit-birth case.  The essential use of the geometric
assumption now enters through the tail identity
\begin{equation}
        q_i^{r-1} = \mathbb P(K_i\ge r)= T_{i,r},
        \qquad r\ge1.
        \label{eq:S-geometric-tail}
\end{equation}
 For $p_i=1$, this identity is understood with $q_i^0=1$: namely,
$T_{i,1}=1$ and $T_{i,r}=0$ for $r\ge2$. The equality \eqref{eq:S-geometric-tail}
allows the sum in \eqref{eq:S-bilinear-geometric} to be evaluated using
the marginal cut-balance identity \eqref{eq:S-cut}.  
Specifically,  reindexing \eqref{eq:S-bilinear-geometric} by setting $j=n+k$ (which is justified since all terms are non-negative) and then using the identity \eqref{eq:S-geometric-tail} yields
\begin{align}
S_i &= p_i \sum_{j\ge1} m_i(j) \sum_{n=0}^{j-1} \pi_i(n)m_i(n)q_i^{j-n-1} \tag{reindexing}\\
&= p_i \sum_{j\ge1} m_i(j) \sum_{n=0}^{j-1} \pi_i(n)m_i(n) T_{i,j-n} \tag{Using \eqref{eq:S-geometric-tail}}\\
&= p_i \sum_{j\ge1} m_i(j)   \frac{j\pi_i(j)}{\tau_i}  \tag{marginal cut-balance identity \eqref{eq:S-cut}}
\end{align}
Consequently,
\begin{align}
S_i
&=
\frac{p_i}{\tau_i}
\sum_{j\ge1}j\pi_i(j)m_i(j)
\notag\\
&=
\frac{p_i}{\tau_i}
\E_\pi[X_i m_i(X_i)].
\label{eq:S-geometric-collapse}
\end{align}
Note that this quantity is finite because
\begin{align}
\E_\pi[X_i m_i(X_i)] &= \E_\pi[X_i f_i(X_{-i})] \tag{tower property}\\
&\le C_\Lambda
\E_\pi[s(X)(1+s(X))] <\infty\notag
\end{align}
by Assumption~\ref{ass:S-standing}\textnormal{(A1)} and
\textnormal{(A3)}.

At this point, we have 
\begin{align*}
I_{i,+} &\le \frac{p_i }{\tau_i \mu_i}\E_\pi[X_i m_i(X_i)] -\mu_i
\end{align*}
Using
\[
        \E_\pi[m_i(X_i)]=\mu_i
\]
and the mean identity (using that $\kappa_i = 1/p_i$)
\[
        \mu_i=\frac{p_i}{\tau_i}\E_\pi[X_i],
\]
we may write
\[
        \mu_i
        =
        \frac{p_i}{\tau_i\mu_i}
        \E_\pi[X_i]\E_\pi[m_i(X_i)].
\]
Therefore,
\begin{align}
I_{i,+}
&\le
\frac{p_i}{\tau_i\mu_i}
\left(
\E_\pi[X_i m_i(X_i)]
-
\E_\pi[X_i]\E_\pi[m_i(X_i)]
\right)
\notag\\
&=
\frac{p_i}{\tau_i\mu_i}
\Cov_\pi(m_i(X_i),X_i).
\end{align}
By the tower property,
\begin{equation}
        \Cov_\pi(m_i(X_i),X_i)
        =
        \Cov_\pi(f_i(X_{-i}),X_i),
        \label{eq:S-conditional-covariance}
\end{equation}
and by \eqref{eq:S-geometric-baseline} the Fano identity
\eqref{eq:S-fano} becomes
\[
        \frac{\Cov_\pi(f_i(X_{-i}),X_i)}{\mu_i}
        =
        F_{X_i}-\kappa_i.
\]
Combining the above with $p_i=1/\kappa_i$, and recalling that $I_{i,-}\le0$,  we  conclude
\[
(\dot I)_{X_i}
\le \frac{1}{\tau_i\kappa_i}
\bigl(F_{X_i}-\kappa_i\bigr),
\]
which was the desired result.
\end{proof}

\begin{proof}[Proof of Theorem~\ref{thm:geometric-main}]
Summing \eqref{eq:S-geometric-local} over
$i\in\{1,\dots,N\}$ gives
\[
        \sum_{i=1}^N(\dot I)_{X_i}
        \le
        \sum_{i=1}^N
        \frac{F_{X_i}-\kappa_i}{\tau_i\kappa_i}.
\]
By Lemma \ref{lem:S-flow-cancellation}
\[
        \sum_{i=1}^N(\dot I)_{X_i}=0.
\]
Therefore
\[
        \sum_{i=1}^N
        \frac{F_{X_i}-\kappa_i}{\tau_i\kappa_i}
        \ge0,
\]
which is precisely \eqref{eq:geom-tradeoff-main}.

Since every coefficient $1/(\tau_i\kappa_i)$ is strictly positive, if
$F_{X_k}<\kappa_k$ for some $k$ and
$F_{X_i}\le\kappa_i$ for every $i\ne k$, then the weighted sum in
\eqref{eq:geom-tradeoff-main} would be strictly negative.  Therefore, whenever
one component lies below its geometric-burst baseline, at least one other
component lies strictly above its own.  This completes the proof of
Theorem~\ref{thm:geometric-main}.
\end{proof}

\section{Two-component counterexamples with nongeometric bursts}
\label{sec:S-counterexamples}

We now prove the expansion asserted in \eqref{eq:analytic-counter-fano-main} for the model \eqref{eq:counter-burst-main}--\eqref{eq:analytic-counter-rates-main}. Fix $r\in(1/2,1]$.  We take $N=2$, set $\tau_i=1$ for $i\in\{1,2\}$, and let
\begin{equation}
        \mathbb P(K_i=1)=1-r,
        \qquad
        \mathbb P(K_i=2)=r,
    \label{eq:S-counter-burst-law}
\end{equation}
which yields
\[
        \E[K_i]=1+r,
        \qquad
        \E[K_i^2]=1+3r.
\]
Thus, we define the common $B_i$ as
\begin{equation}
        B(r)
        :=
        \frac{\E[K_i^2]+\E[K_i]}
        {2\E[K_i]}
        =
        \frac{1+2r}{1+r}.
\label{eq:S-counter-baseline}
\end{equation}

  Define
\begin{equation}
a(n)
=
\begin{cases}
1, & n=0,\\
\frac65, & n=1,\\
\frac95, & n\ge2,
\end{cases}
\qquad
b(n)
=
\begin{cases}
1, & n=0,\\
\frac45, & n=1,\\
\frac15, & n\ge2,
\end{cases}
\label{eq:S-counter-a-b}
\end{equation}
and, for $\varepsilon>0$, we set
\begin{equation}
        f_1^{(\varepsilon)}(n)=\varepsilon a(n),
        \qquad
        f_2^{(\varepsilon)}(n)=\varepsilon b(n).
\label{eq:S-counter-rates}
\end{equation}
These rates are bounded and strictly positive.  Therefore, for every
$\varepsilon>0$, the  conditions of Theorem~\ref{thm:S-sufficient-conditions} are met, and so Assumption \ref{ass:S-standing} holds.

\begin{proposition}
\label{prop:S-analytic-counterexample}
Let $\pi_\varepsilon$ denote the stationary distribution of the model
\eqref{eq:S-counter-burst-law}--\eqref{eq:S-counter-rates}.  Then, for each $i \in \{1,2\}$,
\begin{equation}
        F_{X_i} =  B(r) - \frac{2r(2r-1)}{75}\varepsilon^2+ O(\varepsilon^3), \qquad \text{as} \quad \varepsilon\downarrow0.
\label{eq:S-counter-fano-expansion}
\end{equation}
Consequently, for every fixed $r\in(1/2,1]$, both components lie strictly
below their common constant-rate burst baseline for all sufficiently small
$\varepsilon>0$.
\end{proposition}

\begin{proof}
The proof has two elementary steps.  First, by a stochastic domination argument we show that, in stationarity,
the states outside
\[
        S_4 :=  \left\{(x,y)\in\mathbb Z_{\ge0}^2:x+y\le4 \right\},
\]
contribute only $O(\varepsilon^3)$ to
the probabilities and moments needed below.  This follows by coupling the
total copy number (the sum of the components) to a one-dimensional immigration--death process with a fixed immigration size 2.  Second,
we restrict the stationary balance equations to $S_4$.  They form a finite
linear system whose coefficients depend linearly on $\varepsilon$.
The coefficient matrix at $\varepsilon=0$ is invertible, and hence the
stationary probabilities on $S_4$ admit an expansion through second order
\cite{kato2012short}.  The coefficients are then obtained by solving three
finite linear systems (those for orders $\varepsilon^0$, $\varepsilon^1$, $\varepsilon^2$).

  Let
\[
        Z:=X_1+X_2.
\]
The total production-event rate satisfies
\[
        \lambda_Z  :=  \varepsilon\bigl(a(X_2)+b(X_1)\bigr)  \le    \bar\lambda_\varepsilon := \frac{14}{5}\varepsilon,
\]
while the total degradation rate is exactly $Z$. Note that 
every production event increases $Z$ by either one or two.

Using the same split-coupling construction as in
Proposition~\ref{prop:S-burst-semigroup-monotonicity} (i.e., shared transitions with minimum rate), we couple $Z$ with
the one-dimensional process $Y_\varepsilon$ having transitions
\[
        y\longrightarrow y+2  \quad\text{at constant rate }\bar\lambda_\varepsilon,  \qquad  y\longrightarrow y-1
        \quad\text{at rate }y.
\]
Suppose that initially $Z\le Y_\varepsilon$.  At rate $\lambda_Z$, a
common production event occurs: the original chain undergoes its actual
production event, so $Z$ increases by either one or two, while
$Y_\varepsilon$ increases by two.  At the remaining rate
$\bar\lambda_\varepsilon-\lambda_Z$, only $Y_\varepsilon$ increases.
For degradation, common deaths occur at rate $Z$ (since it is the minimum), decreasing both totals
by one, while the remaining deaths, at rate $Y_\varepsilon-Z$, occur only
in $Y_\varepsilon$.  Each possible coupled transition preserves the
inequality $Z\le Y_\varepsilon$.

Starting both processes from zero therefore gives
\[
        Z(t)\le Y_\varepsilon(t)
        \qquad\text{for all }t\ge0
\]
almost surely.  Since both processes converge to their unique stationary laws, the stationary distribution of $X_1+X_2$ is stochastically dominated by the stationary
distribution of $Y_\varepsilon$.  That is, for every nondecreasing function
$\phi$ for which the expectations are finite,
\[
        \E_{\pi_\varepsilon}\!\left[\phi(X_1+X_2)\right]
        \le
        \E\!\left[\phi(Y_\varepsilon)\right],
\]
where $Y_\varepsilon$ on the right is distributed according to its stationary
law.  See, for example, \cite{shaked2007stochastic}.

We turn to characterizing the stationary distribution of $Y_\varepsilon$.  Let $G_\varepsilon$ denote the  probability generating function for the stationary distribution
of $Y_\varepsilon$.  For $0<z<1$, its stationary generator identity gives
\[
        0   =   \bar\lambda_\varepsilon(z^2-1)G_\varepsilon(z)+   (1-z)G_\varepsilon'(z).
\]
Together with $G_\varepsilon(1)=1$, this yields
\begin{equation}
        G_\varepsilon(z)
        =
        \exp\left\{
        \frac{7}{5}\varepsilon(z^2-1)
        +
        \frac{14}{5}\varepsilon(z-1)
        \right\}.
\label{eq:S-dominating-pgf}
\end{equation}
Hence,
\[
        Y_\varepsilon
        \overset{d}{=}
        2A_\varepsilon+C_\varepsilon,
\]
where
\[
        A_\varepsilon
        \sim
        \operatorname{Poisson}\left(\frac75\varepsilon\right),
        \qquad
        C_\varepsilon
        \sim
        \operatorname{Poisson}\left(\frac{14}{5}\varepsilon\right)
\]
are independent.

Let
\[
        N_\varepsilon:=A_\varepsilon+C_\varepsilon.
\]
Then
\[
        N_\varepsilon
        \sim
        \operatorname{Poisson}\left(\frac{21}{5}\varepsilon\right),
        \qquad
        Y_\varepsilon\le2N_\varepsilon.
\]
The event $Y_\varepsilon\ge5$ requires at least three Poisson arrivals,
and hence
\[
        \{Y_\varepsilon\ge5\}
        \subseteq
        \{N_\varepsilon\ge3\}.
\]
Therefore, for $q\in\{0,1,2\}$,
\begin{align*}
\E\left[
        Y_\varepsilon^q
        \mathbf 1_{\{Y_\varepsilon\ge5\}}
\right]
&\le
2^q
\E\left[
        N_\varepsilon^q
        \mathbf 1_{\{N_\varepsilon\ge3\}}
\right]
\\
&=
O(\varepsilon^3).
\end{align*}
Applying stochastic domination with the nondecreasing function
$\phi(z)=z^q\mathbf 1_{\{z\ge5\}}$ gives
\begin{equation}
        \E_{\pi_\varepsilon}
        \left[
        Z^q\mathbf 1_{\{Z\ge5\}}
        \right]
        =
        O(\varepsilon^3),
        \qquad
        q\in\{0,1,2\}.
\label{eq:S-counter-S4-tail}
\end{equation}
Thus the states outside $S_4$ contribute only $O(\varepsilon^3)$ to all
probabilities and moments used below.

Now write
\[
        \pi_\varepsilon(x,y)
        :=
        \mathbb P_{\pi_\varepsilon}(X_1=x,X_2=y),
        \qquad
        x,y\in\mathbb Z_{\ge0}.
\]
The stationary balance equation at $(x,y)$ is
\begin{align}
\begin{split}
\left[
x+y+\varepsilon\bigl(a(y)+b(x)\bigr)
\right]
\pi_\varepsilon(x,y)
&=
(x+1)\pi_\varepsilon(x+1,y)
+
(y+1)\pi_\varepsilon(x,y+1)\\
&\quad+
\varepsilon a(y)
\left[
(1-r)\pi_\varepsilon(x-1,y)
+
r\pi_\varepsilon(x-2,y)
\right]\\
&\quad+
\varepsilon b(x)
\left[
(1-r)\pi_\varepsilon(x,y-1)
+
r\pi_\varepsilon(x,y-2)
\right],
\end{split}
\label{eq:S-counter-stationary-balance}
\end{align}
where probabilities with a negative coordinate are interpreted as zero.

We now restrict the stationary calculation to $S_4$. Let
\[
        \mathbf p_\varepsilon
        :=
        \bigl(
        \pi_\varepsilon(x,y)
        \bigr)_{(x,y)\in S_4}.
\]
This vector has fifteen components, and so we set up 15 new equations.  For each of the fourteen states
$(x,y)\in S_4\setminus\{(0,0)\}$, we use the stationary balance equation
\eqref{eq:S-counter-stationary-balance} evaluated at that state.  As the
fifteenth equation, we use normalization:
\[
        \sum_{(x,y)\in S_4}\pi_\varepsilon(x,y)
        =
        1-\mathbb P_{\pi_\varepsilon}(Z\ge5).
\]
If $1\le x+y\le3$, every probability appearing in the corresponding
balance equation belongs to a state in $S_4$.  If $x+y=4$, the incoming
degradation terms come from states with total count five and are therefore
placed in the remainder vector.

We index both the components of $\mathbf p_\varepsilon$ and the rows of the
system by the states in $S_4$.  For $(x,y)\ne(0,0)$, the $(x,y)$ row is
the stationary balance equation at $(x,y)$; the $(0,0)$ row is the
normalization equation.  Collecting these equations gives
\begin{equation}
        (A_0+\varepsilon A_1)\mathbf p_\varepsilon
        =
        \mathbf d+\boldsymbol\rho_\varepsilon,
\label{eq:S-counter-finite-system}
\end{equation}
where
\[
        \mathbf d_{x,y}
        =
        \mathbf 1_{\{(x,y)=(0,0)\}}.
\]
The vector $\boldsymbol\rho_\varepsilon$ consists precisely of
the missing stationary mass in the normalization equation and the incoming
degradation flux from states with total count five.  Hence
\[
        \|\boldsymbol\rho_\varepsilon\|
        =
        O(\varepsilon^3)
\]
by \eqref{eq:S-counter-S4-tail}.

The matrix $A_0$ is invertible.  Indeed, at $\varepsilon=0$ the
non-normalization rows are the stationary balance equations for the
pure-death chain.  Solving these equations successively from level
$x+y=4$ down to level $x+y=1$ shows that every solution is supported at
$(0,0)$; the homogeneous normalization equation then forces that remaining
entry to vanish.  Thus $\ker A_0=\{0\}$.
Standard finite-dimensional analytic perturbation
theory therefore gives, uniformly over $(x,y)\in S_4$,
\begin{equation}
        \pi_\varepsilon(x,y)
        =
        c_0(x,y)
        +
        \varepsilon c_1(x,y)
        +
        \varepsilon^2c_2(x,y)
        +
        O(\varepsilon^3);
\label{eq:S-counter-pi-expansion}
\end{equation}
see, for example, Ref.~\cite[Chap.~II]{kato2012short}.

Substituting this expansion into
\eqref{eq:S-counter-finite-system} and using
$\boldsymbol\rho_\varepsilon=O(\varepsilon^3)$ gives
\[
        A_0\mathbf c_0=\mathbf d,
        \qquad
        A_0\mathbf c_1=-A_1\mathbf c_0,
        \qquad
        A_0\mathbf c_2=-A_1\mathbf c_1,
\]
where
\[
        \mathbf c_\ell
        :=
        \bigl(c_\ell(x,y)\bigr)_{(x,y)\in S_4}.
\]
Since $A_0$ is invertible, these three finite systems uniquely determine
the coefficients through second order.  

Solving these finite systems and substituting the resulting coefficients into the stationary expectations gives
\begin{align}
\E_{\pi_\varepsilon}[a(X_2)]
&=
1+\frac{1+2r}{5}\varepsilon
+O(\varepsilon^2),
\label{eq:S-counter-Ea}
\\
\E_{\pi_\varepsilon}[X_1a(X_2)]
&=
(1+r)\varepsilon
+
\frac{56r^2+92r+30}{75}\varepsilon^2
+
O(\varepsilon^3),
\label{eq:S-counter-EX1a}
\\
\E_{\pi_\varepsilon}[b(X_1)]
&=
1-\frac{1+2r}{5}\varepsilon
+O(\varepsilon^2),
\label{eq:S-counter-Eb}
\\
\E_{\pi_\varepsilon}[X_2b(X_1)]
&=
(1+r)\varepsilon
-
\frac{64r^2+88r+30}{75}\varepsilon^2
+
O(\varepsilon^3).
\label{eq:S-counter-EX2b}
\end{align}

The exact stationary first-moment balances are
\[
        \E_{\pi_\varepsilon}[X_1]
        =
        (1+r)\varepsilon
        \E_{\pi_\varepsilon}[a(X_2)],
        \qquad
        \E_{\pi_\varepsilon}[X_2]
        =
        (1+r)\varepsilon
        \E_{\pi_\varepsilon}[b(X_1)].
\]
Consequently,
\begin{align}
\E_{\pi_\varepsilon}[X_1]
&=
(1+r)\varepsilon
+
\frac{(1+r)(1+2r)}{5}\varepsilon^2
+
O(\varepsilon^3),
\label{eq:S-counter-EX1}
\\
\E_{\pi_\varepsilon}[X_2]
&=
(1+r)\varepsilon
-
\frac{(1+r)(1+2r)}{5}\varepsilon^2
+
O(\varepsilon^3).
\label{eq:S-counter-EX2}
\end{align}

It follows that
\begin{align}
\Cov_{\pi_\varepsilon}(a(X_2),X_1)
&=
\E_{\pi_\varepsilon}[X_1a(X_2)]
-
\E_{\pi_\varepsilon}[X_1]
\E_{\pi_\varepsilon}[a(X_2)]
\notag\\
&=
-\frac{2r(2r-1)}{75}\varepsilon^2
+
O(\varepsilon^3),
\label{eq:S-counter-cov1}
\end{align}
and similarly,
\begin{equation}
        \Cov_{\pi_\varepsilon}(b(X_1),X_2)
        =
        -\frac{2r(2r-1)}{75}\varepsilon^2
        +
        O(\varepsilon^3).
\label{eq:S-counter-cov2}
\end{equation}

Finally,
\[
        f_1^{(\varepsilon)}(X_2)
        =
        \varepsilon a(X_2),
        \qquad
        \mu_1
        =
        \varepsilon
        \E_{\pi_\varepsilon}[a(X_2)].
\]
Therefore, the general Fano identity \eqref{eq:S-fano} gives
\begin{align*}
F_{X_1}-B(r)
&=
\frac{
\Cov_{\pi_\varepsilon}
\bigl(
f_1^{(\varepsilon)}(X_2),X_1
\bigr)
}{\mu_1}
\\
&=
\frac{
\Cov_{\pi_\varepsilon}(a(X_2),X_1)
}{
\E_{\pi_\varepsilon}[a(X_2)]
}.
\end{align*}
Using \eqref{eq:S-counter-Ea} and
\eqref{eq:S-counter-cov1}, we obtain
\[
        F_{X_1}-B(r)
        =
        -\frac{2r(2r-1)}{75}\varepsilon^2
        +
        O(\varepsilon^3).
\]
The same argument, using \eqref{eq:S-counter-Eb} and
\eqref{eq:S-counter-cov2}, gives
\[
        F_{X_2}-B(r)
        =
        -\frac{2r(2r-1)}{75}\varepsilon^2
        +
        O(\varepsilon^3).
\]
This proves \eqref{eq:S-counter-fano-expansion}.

For $r>1/2$, the leading coefficient is strictly negative.  Hence both
Fano factors are strictly below $B(r)$ for all sufficiently small positive
$\varepsilon$.
\end{proof}

\begin{remark}
The preceding example is not finely tuned.  More generally, replace
\eqref{eq:S-counter-a-b} by
\[
a(0)=b(0)=1,
\qquad
a(1)=1+c,
\qquad
b(1)=1-c,
\]
and
\[
a(n)=1+d,
\qquad
b(n)=1-d,
\qquad n\ge2.
\]
The same finite calculation gives, for $j=1,2$,
\begin{equation}
F_{X_j}-B(r)
=
\frac{
r(2c-d)(2c+dr-d)
}{6}
\varepsilon^2
+
O(\varepsilon^3).
\label{eq:S-general-counter-family}
\end{equation}
Thus, whenever
\[
        0<c<\frac12,
        \qquad
        2c<d<1,
        \qquad
        r>1-\frac{2c}{d},
\]
the leading coefficient in
\eqref{eq:S-general-counter-family} is strictly negative.  These strict
conditions define an open set of rate and burst-law parameters.  The
choice
\[
        c=\frac15,
        \qquad
        d=\frac45
\]
used above makes the burst threshold exactly $r>1/2$.
\end{remark}

\section{Termwise bounds under signed monotonicity}
\label{sec:S-signed-monotonicity}

Recall the signed order \eqref{eq:signed-order-main} and the
signed-monotonicity condition \eqref{eq:signed-monotonicity-main}.
A probability distribution $\rho$ on $S$ is said to be associated with
respect to $\preceq_\sigma$ if
\[
        \Cov_\rho(u,v)\ge0
\]
for every pair of bounded $\preceq_\sigma$--increasing functions
$u,v:S\to\mathbb R$.

The proof of Theorem~\ref{thm:signed-main} follows the four-step argument
of \cite[Section~5.3]{anderson2026noise}.  Most of that argument applies
without change.  We give details for the two points that require
modification for arbitrary burst sizes: common production events in the
order-preserving coupling must use the same burst-size realization, and
the finite-state approximation must cap bursts at the boundary rather
than suppress them.

Let $(P_t)_{t\ge0}$ denote the Markov semigroup,
\[
        (P_tu)(x):=\E_x[u(X(t))].
\]

\begin{proposition}[Monotonicity of the burst semigroup]
\label{prop:S-burst-semigroup-monotonicity}
Suppose Assumption~\ref{ass:S-standing} holds and the network satisfies
\eqref{eq:signed-monotonicity-main}.  Then $P_t$ is monotone with respect
to $\preceq_\sigma$ for every $t\ge0$: if $u$ is bounded and
$\preceq_\sigma$--increasing, then $P_tu$ is also
$\preceq_\sigma$--increasing.
\end{proposition}

\begin{proof}
Fix $x\preceq_\sigma y$.  We construct two copies
$(X^x,X^y)$, started from $x$ and $y$, that preserve this order.

Suppose the current states are $z\preceq_\sigma w$.  For production in
coordinate $i$, set
\[
\begin{aligned}
b_i^0(z,w)
&:=
\min\{f_i(z_{-i}),f_i(w_{-i})\},\\
b_i^z(z,w)
&:=
f_i(z_{-i})-b_i^0(z,w),\\
b_i^w(z,w)
&:=
f_i(w_{-i})-b_i^0(z,w).
\end{aligned}
\]
At rate $b_i^0(z,w)$, a single random variable with the law of $K_i$ is
sampled and the same burst size is added to coordinate $i$ of both
processes.  At rates $b_i^z(z,w)$ and $b_i^w(z,w)$, a burst with the law
of $K_i$ is added only to the corresponding process.

If $\sigma_i=1$, then signed monotonicity gives
\[
        f_i(z_{-i})\le f_i(w_{-i}),
\]
so $b_i^z(z,w)=0$: an unmatched burst can occur only in $X^y$, where it
preserves $z_i\le w_i$.  If $\sigma_i=-1$, then
\[
        f_i(z_{-i})\ge f_i(w_{-i}),
\]
so $b_i^w(z,w)=0$: an unmatched burst can occur only in $X^x$, where it
preserves $z_i\ge w_i$.  A common burst preserves the order because the
same burst size is added to both coordinates.

Death events are coupled by the same split construction, with common
rate
\[
        \min\left\{\frac{z_i}{\tau_i},\frac{w_i}{\tau_i}\right\}.
\]
Any residual death occurs only in the process with the larger
$i$th coordinate.  Since deaths have unit size and the coordinates are
integer valued, such an event cannot reverse the corresponding signed
inequality.  Thus every possible coupled jump preserves
$\preceq_\sigma$.

The marginals have the correct transition rates, and the coupled process
is nonexplosive because its marginals are nonexplosive.  Hence
\[
        X^x(t)\preceq_\sigma X^y(t)
        \qquad\text{for every }t\ge0
\]
almost surely.  Therefore, for every bounded
$\preceq_\sigma$--increasing function $u$,
\[
        (P_tu)(x)
        \le
        (P_tu)(y),
\]
which proves the proposition.
\end{proof}

\begin{proof}[Proof of Theorem~\ref{thm:signed-main}]
It remains to obtain association of the stationary law and apply it to
the relevant observables.  We first introduce the finite-state
approximation needed for the association argument.

For $M\ge1$, let
\[
        S_M:=\{0,1,\dots,M\}^N.
\]
For $z\in S_M$, define
\[
        \Theta_{i,k}^{(M)}(z)
        :=
        z+
        \bigl((z_i+k)\wedge M-z_i\bigr)e_i.
\]
Thus $\Theta_{i,k}^{(M)}(z)$ is obtained by adding a burst of size $k$ in
coordinate $i$ and then capping that coordinate at $M$.  Define
\begin{align}
(L^{(M)}u)(z)
&:=   \sum_{i=1}^N
f_i(z_{-i})
\sum_{k\ge1}p_{i,k}
\bigl[
u(\Theta_{i,k}^{(M)}(z))-u(z)
\bigr] +
\sum_{i=1}^N
\frac{z_i}{\tau_i}
\bigl[
u(z-e_i)-u(z)
\bigr],
\label{eq:S-capped-generator}
\end{align}
where the death term is omitted when $z_i=0$, and let
$(P_t^{(M)})_{t\ge0}$ be the corresponding semigroup.

The split coupling from
Proposition~\ref{prop:S-burst-semigroup-monotonicity} also applies to
this capped process.  Indeed, if $\sigma_i=1$ and $z_i\le w_i$, then
\[
        (z_i+k)\wedge M
        \le
        (w_i+k)\wedge M,
\]
whereas if $\sigma_i=-1$ and $z_i\ge w_i$, then
\[
        (z_i+k)\wedge M
        \ge
        (w_i+k)\wedge M.
\]
Thus common capped bursts preserve the signed order, unmatched bursts
occur only in the order-preserving process, and the death argument is
unchanged.  Consequently, $P_t^{(M)}$ is monotone with respect to
$\preceq_\sigma$.

Having established monotonicity of the capped semigroup, the remainder of
the association argument is unchanged from
\cite[Sections~5.3.2--5.3.4]{anderson2026noise}.  With
$L^{(M)}$ and $P_t^{(M)}$ denoting the capped generator and semigroup,
respectively, define $\Gamma_M$, $a_s$, $b_s$, $F_s=a_sb_s$, and
\[
        H(s):=P_s^{(M)}F_s(x)
\]
exactly as in that reference.  Since every nontrivial capped transition
changes only one coordinate, its initial and terminal states are comparable
under $\preceq_\sigma$.  Consequently,
$\Gamma_M(a_s,b_s)\ge0$, and the same semigroup-interpolation calculation
shows that the transition laws of the capped process are associated.

The passage to the full process is also unchanged.  Couple the capped and
full processes using the same event times and burst-size realizations.  They
agree until the first time the full process exits $S_M$.  By nonexplosion and
the fact that every burst size is finite almost surely, these exit times tend
to infinity as $M\to\infty$.  Hence the transition laws of the full process
are associated.  The stationary and truncation arguments of
\cite[Sections~5.3.3--5.3.4]{anderson2026noise} then imply
\[
        \Cov_\pi\bigl(
        \sigma_i f_i(X_{-i}),\sigma_iX_i
        \bigr)
        =
        \Cov_\pi\bigl(f_i(X_{-i}),X_i\bigr)
        \ge0.
\]
Finally, \eqref{eq:S-fano} and $\mu_i>0$ give
\[
        F_{X_i}-B_i
        =
        \frac{\Cov_\pi(f_i(X_{-i}),X_i)}{\mu_i}
        \ge0,
\]
which proves \eqref{eq:signed-termwise-main}.
\end{proof}

The structural-balance criterion stated in the main text is purely
graph-theoretic and is unchanged from
\cite[Section~5.1]{anderson2026noise}.

\end{document}